\documentclass[10pt,a4paper]{article}
\usepackage[T1]{fontenc}
\usepackage[utf8]{inputenc}
\usepackage{authblk}
\usepackage{amsmath,amssymb,amsthm,esint,bm}
\usepackage{mathrsfs}
\usepackage{bookmark}
\usepackage{amsmath}
\allowdisplaybreaks[3]

\newtheorem{definition}{Definition}[section]
\newtheorem{theorem}[definition]{Theorem}
\newtheorem{lemma}[definition]{Lemma}

\theoremstyle{remark}

\numberwithin{equation}{section}

\title{Wolff Type Potential Estimates for Stationary Stokes Systems with Dini-$BMO$ Coefficients}

\author[a]{Lingwei Ma}
\author[b]{Zhenqiu Zhang\thanks{Corresponding author.}}

\affil[a]{School of Mathematical Sciences, Nankai University, Tianjin 300071, P.R. China}
\affil[b]{School of Mathematical Sciences and LPMC, Nankai University, Tianjin, 300071, P.R. China}

\date{\today}

\usepackage{hyperref}
\begin{document}
\maketitle
\footnotetext[1]{E-mail: 1120170026@mail.nankai.edu.cn (L. Ma), zqzhang@nankai.edu.cn (Z. Zhang).}
\begin{abstract}
The pointwise gradient estimate for weak solution pairs to the stationary Stokes system with Dini-$BMO$ coefficients is established via the Havin-Maz'ya-Wolff type nonlinear potential of the nonhomogeneous term. In addition, we present a pointwise bound for the weak solutions under no extra regularity assumption on the coefficients.\\



\end{abstract}


\section{Introduction.}\label{section1}

In this paper, the following generalized stationary Stokes problem
\begin{equation}\label{model}
\left\{\begin{array}{r@{\ \ }c@{\ \ }ll}
\operatorname{div}\left({\bf A}(x)D {\bf u}\right)-\nabla\pi & =& \operatorname{div}{\bf F} & \mbox{in}\ \ \Omega\,, \\[0.05cm]
\operatorname{div}{\bf u}&=& 0 & \mbox{in}\ \ \Omega\,, \\[0.05cm]
{\bf u} &=& {\bf 0} & \mbox{on}\ \ \partial\Omega\,,
\end{array}\right.
\end{equation}
is considered in a bounded domain $\Omega\subset\mathbb{R}^{n}$ ($n\geq2$), where ${\bf F}\in L^{2}(\Omega\,,\mathbb{R}^{n\times n})$ is a given exterior force and the coefficient ${\bf A}=\left\{A_{ij}^{kl}\right\}$ with $1\leq i,j,k,l\leq n$ is assumed to be uniformly elliptic and bounded, i.e., there exist constants $0<\lambda\leq 1\leq \Lambda<\infty$ such that
\begin{equation}\label{uniformellipticity}
\lambda\,|\xi|^{2}\,\leq\,\left\langle\,{\bf A}(x)\, \xi\,, \xi\,\right\rangle\,, \qquad\left\langle\,{\bf A}(x)\, \xi\,, \eta\,\right\rangle\,\leq\,\Lambda\,|\xi||\eta|
\end{equation}
for almost every $x\in\Omega$ and every $\xi,\,\eta\in\mathbb{R}^{n\times n}$\,, where $\langle\cdot,\cdot\rangle$ denotes the standard inner product in $\mathbb{R}^{n\times n}$. More precise regularity assumptions on the vector field ${\bf A}$ will be postponed in this section. The unknowns are the velocity of the fluid flow ${\bf u}:\Omega\rightarrow\mathbb{R}^{n}$ and its pressure $\pi:\Omega\rightarrow\mathbb{R}$. Here the second equation $\operatorname{div}{\bf u}= 0$ reveals that the fluid is homogeneous and incompressible.

System \eqref{model} describes the flow of Newtonian fluid, which
serves as a typical model in fluid mechanics (cf. \cite{BMR}). In continuum mechanics, Newtonian fluid means that the viscous stresses at every point are
linearly proportional to the local strain rate. For the existence of a
weak solution pair to Stokes system \eqref{model}, we refer the reader to
\cite{ChLe, La}\,. Our interest here is to establish the pointwise estimates for both weak solution pairs and their gradients to \eqref{model} by using the Havin-Maz'ya-Wolff type nonlinear potential of the nonhomogeneous term.

Potential theory plays an essential role in the regularity theory of partial differential equations.
During the last decade considerable attentions have been paid to the investigation of the pointwise potential type estimates for both the solutions and their gradients to non-Newtonian field model whose nonhomogeneous term is a function or a more general measure, due to these pointwise estimates provide a unified approach to get the norm bounds for solutions in a wide range of functions spaces. As a consequence of these results, some regularity properties for these solutions can be established, such as H\"{o}lder continuity, Calder\'on–Zygmund estimates and so on. The pioneering works on this subject originate from Kilpel\"{a}inen and Mal\'{y} \cite{KiMa, KiMa2}, who proved the pointwise estimates for solutions to the nonlinear equations of $p$-Laplace type via the Wolff type nonlinear potentials. Here the Havin-Maz'ya-Wolff potential was introduced by Maz'ya and Havin \cite{MaHa} and the relevant fundamental contributions were attributed to Hedberg and Wolff \cite{HeWo}. Later Trudinger and Wang \cite{TrWa} provided these estimates in terms of a different method. Besides, some important
results have been achieved by Labutin \cite{Lab} concerning the Wolff type potential estimates related to fully nonlinear Hessian type operators.
With respect to the pointwise bounds for solutions to the $p$-Laplace system, we refer the reader to \cite{CiSc}. The study of this subject has received an impulse by Mingione \cite{Min}, who showed that the pointwise gradient estimates for solutions to the quasilinear equations of Laplace type employing the linear Riesz potentials, exactly as it happens for the Poisson equation via representation formulas.
Subsequently, Duzaar, Kuusi and Mingione \cite{DuMi, DuMi2, DuMi3, KuMi, KuMi2, KuMi3} obtained the similar pointwise bounds also hold for the gradient of solutions to the nonlinear equations and systems via the Wolff type nonlinear potentials.

In contrast, much less effort has been devoted to the pointwise potential estimates for both the solutions and their gradients to Newtonian fluid to our knowledge, so
we are naturally led to consider such estimates for \eqref{model}. For the sake of stating our main results, we need to present some definitions, notations and assumptions. Let us begin with the definition of the weak solution pair to problem \eqref{model}\,.
\begin{definition}\label{Def-weaksolution}
Let ${\bf F}\in L^{2}(\Omega\,,\mathbb{R}^{n\times n})$\,. Then ${\bf u}(x)$ is called a weak solution of the Stokes systems \eqref{model}, if
\begin{equation*}
{\bf u}\in W^{1,2}_{0,\sigma}\left(\Omega,\mathbb{R}^{n}\right):=\left\{{\bf v}\in W^{1,2}_{0}\left(\Omega,\mathbb{R}^{n}\right)|\operatorname{div}{\bf v}=0\right\}
\end{equation*}
and solves \eqref{model} in the distribution sense, i.e.,
\begin{equation*}
 \int_{\Omega}\big\langle{\bf
A}(x){\bf \bf\nabla u}\,,{\bf\bf\nabla\phi}\big\rangle\operatorname{d}\!x
= \int_{\Omega}\big\langle{\bf F}\,,{\bf
\bf\nabla\phi}\big\rangle\operatorname{d}\!x
\end{equation*}
for any divergence free test function ${\bf \phi}\in
W_{0,\sigma}^{1,2}(\Omega,\mathbb{R}^{n})$\,. Meanwhile, if ${\bf u}$ is such a weak solution and $\pi\in L^{2}(\Omega)$ stands for an associated pressure of ${\bf u}$, which satisfies
\begin{equation*}
 \int_{\Omega}\big\langle{\bf
A}(x){\bf \bf\nabla u}\,,{\bf\bf\nabla\varphi}\big\rangle-\pi \operatorname{div}{\bf\varphi}\operatorname{d}\!x =
\int_{\Omega}\big\langle{\bf F}\,,{\bf
\bf\nabla\varphi}\big\rangle\operatorname{d}\!x
\end{equation*}
for any test function ${\bf \varphi}\in
W_{0}^{1,2}(\Omega,\mathbb{R}^{n})$\,, then $\left({\bf u},\pi\right)$ is called a weak solution pair of \eqref{model}.
\end{definition}

Next, we introduce the definition of the truncated Havin-Maz'ya-Wolff potential.

\begin{definition}
Let $s>1$ and $\alpha>0$, the truncated Havin-Maz'ya-Wolff potential ${\bf W}_{\alpha,s}^{R}f$ of $f\in L_{\rm loc}^{1}(\Omega)$ is defined by
\begin{equation*}
  {\bf W}_{\alpha,s}^{R}f(x_{0})=\int_{0}^{R}\left(\varrho^{\alpha s}\fint_{B_{\varrho}(x_{0})}| f(x)|\operatorname{d}\!x\right)^{\frac{1}{s-1}}\frac{\operatorname{d}\!\varrho}{\varrho}
\end{equation*}
for any $x_{0}\in\Omega$ and $R>0$ such that $B_{R}(x_{0})\subset\Omega$, where $B_{\varrho}(x_{0})$ denotes an open ball in $\mathbb{R}^{n}$ with center $x_{0}$ and radius $\varrho>0$.
\end{definition}

Throughout this paper we assume that ${\bf A}$ has a small $BMO$ semi-norm and satisfies the Dini-$BMO$ regularity. More precisely,

\begin{definition}\label{Def-Vanish}
We say that the vector field ${\bf A}$ is $\left(\delta, R\right)$-vanishing for some $\delta$, $R>0$, if
\begin{equation}\label{smallBMO}
[{\bf A}]_{\operatorname{BMO}}^{2}(R):=\sup_{{\substack{ x_{0}\,\in\,\Omega\\0<r\leq R}}} \fint_{B_{r}(x_{0})}
\left|{\bf A}(x)-\overline{{\bf A}}_{B_{r}(x_{0})}\right|^{2}\operatorname{d}\!x\leq\delta^{2}\,,
\end{equation}
where
\begin{equation*}
\overline{{\bf A}}_{B_{r}(x_{0})}:=\fint_{B_{r}(x_{0})}{\bf A}(x)\operatorname{d}\!x=\frac{1}{\left|B_{r}(x_{0})\right|}\int_{B_{r}(x_{0})}{\bf A}(x)\operatorname{d}\!x
\end{equation*}
denotes the matrix whose entries are the integral averages for the corresponding entries of ${\bf A}$ over $B_{r}(x_{0})$\,. Besides, we impose a decay property on $\bf A$, which is called that Dini-$BMO$ regularity if
\begin{equation}\label{dini-bmo}
  d(R):=\int_{0}^{R}\left([{\bf A}]_{\operatorname{BMO}}^{2}(\varrho)\right)^{\frac{1}{2}-\frac{1}{q}}\frac{\operatorname{d}\!\varrho}{\varrho}<\infty
\end{equation}
for any $R>0$, where $q=q(\lambda,\Lambda,n)>2$ is given in \eqref{v-reverseholder}.
\end{definition}

In what follows $C$ denotes a constant whose value may be different from line to line, and only the relevant dependence is specified. Now we are in position to state our main result of this paper. The first result infers that the gradient of weak solution $D{\bf u}$ and its pressure $\pi$ to the stationary Stokes system \eqref{model} with Dini-$BMO$ coefficients can be controlled by an unconventional Havin-Maz'ya-Wolff
type nonlinear potential of the nonhomogeneous term $\bf F$.
\begin{theorem}\label{Th1}\textup{(Nonlinear potential gradient estimate)}
Let $({\bf u},\pi)$ be a weak solution pair of
\eqref{model} with ${\bf F}\in L
^{2}\left(\Omega\,,\mathbb{R}^{n\times n}\right)$ and $\bf A$ satisfying \eqref{uniformellipticity}, \eqref{smallBMO} and \eqref{dini-bmo} for some $\delta=\delta(n,\lambda,\Lambda)>0$ and $R>0$. Then there exist a positive constant $C=C(n,\lambda,\Lambda)$ and a positive radius $R_{0}=R_{0}\left(n,\lambda,\Lambda, d(\cdot)\right)$ such that the pointwise estimate
\begin{equation}\label{gradient-estimate}
  |D{\bf u}(x_{0})|+|\pi (x_{0})|\leq C\fint_{B_{R}(x_{0})}\left|D{\bf u}\right|\operatorname{d}\!x+ C\fint_{B_{R}(x_{0})}\left|\pi\right|\operatorname{d}\!x+C\int_{0}^{2R}\left( \fint_{B_{\varrho}(x_{0})}\left(\frac{|{\bf F}-\overline{{\bf F}}_{B_{\varrho}(x_{0})}|}{\varrho}\right)^{2}\operatorname{d}\!x\right)^{\frac{1}{2}}\operatorname{d}\!\varrho
\end{equation}
holds for almost all $x_{0}\in\Omega$ and every $B_{2R}(x_{0})\subset\Omega$ with $R\leq R_{0}$.
\end{theorem}

In analogy with the proof of Theorem \ref{Th1}\,, we can establish a pointwise bound for the weak solution $\bf u$ in terms of the Havin-Maz'ya-Wolff potential of $\bf F$, while no extra regularity assumption other than \eqref{uniformellipticity} is to be assumed on the coefficients.
\begin{theorem}\label{Th2}\textup{(Zero order estimate)}
Let $({\bf u},\pi)$ be a weak solution pair of
\eqref{model} with ${\bf F}\in L
^{2}\left(\Omega\,,\mathbb{R}^{n\times n}\right)$ and $\bf A$ satisfying \eqref{uniformellipticity}. Then there exists a positive constant $C=C(n,\lambda,\Lambda)$ such that the pointwise estimate
\begin{equation}\label{zero-estimate}
  |{\bf u}(x_{0})|\leq C\fint_{B_{R}(x_{0})}\left|{\bf u}\right|\operatorname{d}\!x+C\, {\bf W}_{\frac{2}{3},3}^{2R}\left(|{\bf F}|^{2}\right)(x_{0})
\end{equation}
holds for almost all $x_{0}\in\Omega$ and every $B_{2R}(x_{0})\subset\Omega$.
\end{theorem}

The remainder of this paper is organized as follows. Section \ref{section2} contains some preliminary results, the comparison systems and
the relevant comparison estimates. In the last section, we complete the proof of Theorem \ref{Th1} and \ref{Th2}\,.

\section{Preliminaries and comparison estimates.}\label{section2}

In this section, we first collect some auxiliary results which will be useful in the sequel. Furthermore, we introduce the comparison systems and investigate the relevant comparison estimates.

\subsection{Preliminary results}
 Let us begin with the following two technical lemmas.
\begin{lemma}\label{existence-Johndomain}{\rm (cf. \cite{ADM})}
Let $\Omega\in\mathbb{R}^{n}$, $n\geq2$ be a bounded John domain. Given $g\in L^{p}(\Omega)$, $1<p<+\infty$ such that
$\int_{\Omega}g\operatorname{d}\!x=0$, then there exists ${\bf \psi}\in W_{0}^{1,p}\left(\Omega,\mathbb{R}^{n}\right)$ satisfying
\begin{eqnarray*}
  \operatorname{div}{\bf \psi}&=& g \quad\mbox{in}\ \ \Omega\,,\\
  \|{\bf\nabla\psi}\|_{L^{p}(\Omega)} &\leq& C \|g\|_{L^{p}(\Omega)}\,,
\end{eqnarray*}
where the positive constant $C=C\left(\operatorname{diam}(\Omega),n,p\right)$. Particularly, if $\Omega=B_{R}(x_{0})$, then $C$ depends only on $n$ and $p$.
\end{lemma}

\begin{lemma}\label{uniqueness-Johndomain}{\rm (cf. \cite{ADM})}
Let $\Omega\in\mathbb{R}^{n}$, $n\geq2$ be a bounded John domain.
Then any bounded linear functional $\mathcal{F}$ on
$W_{0}^{1,2}\left(\Omega,\mathbb{R}^{n}\right)$ identically
vanishing on $W_{0,\sigma}^{1,2}\left(\Omega,\mathbb{R}^{n}\right)$
is of the form $\mathcal{F}({\bf
\psi})=\int_{\Omega}\pi\operatorname{div}{\bf
\psi}\operatorname{d}\!x$ for some uniquely determined $\pi\in
L^{2}(\Omega)/\mathbb{R}$.
\end{lemma}

The next result is regarding to a self-improving property of reverse H\"{o}lder inequality.

\begin{lemma}\label{reverse-holder} {\rm (cf. \cite{Giu})}
Let ${\bf g}\in L_{\rm loc}^{1}(\Omega,\mathbb{R}^{n\times n})$. Assume that there exist constants $0<\tau<1$, $\gamma>1$ and $C_{0}>0$ such that
\begin{equation*}
  \left(\fint_{B_{\tau r}}|{\bf g}|^{\gamma}\operatorname{d}\!x\right)^{\frac{1}{\gamma}}\leq C_{0}\fint_{B_{r}}|{\bf g}|\operatorname{d}\!x
\end{equation*}
for every $B_{r}\subset\Omega$. Then
\begin{equation*}
  \left(\fint_{B_{\tau r}}|{\bf g}|^{\gamma}\operatorname{d}\!x\right)^{\frac{1}{\gamma}}\leq C\left(\fint_{B_{r}}|{\bf g}|^{t}\operatorname{d}\!x\right)^{\frac{1}{t}}
\end{equation*}
for every $0<t\leq1$, where the positive constant $C=C(C_{0},n,\tau,t)$.
\end{lemma}

We end this section by providing the following basic estimate, which will be utilized frequently to discuss the oscillation estimates. Let $E$ be a measurable subset in $\mathbb{R}^{n}$. For any ${\bf f}\in L^{p}(E, \mathbb{R}^{ m})$ with $p\in[1,\infty)$ and ${\bf h}\in\mathbb{R}^{ m}$ for $m\geq1$, we have
\begin{equation}\label{eqn-minimal2}
\left(\fint_{E}\left|\,{\bf f}(x)-\overline{{\bf f}}_{E}\,\right|^{p}\operatorname{d}\!x\right)^{\frac{1}{p}}\,\leq\,2\min_{{\bf h}\in\mathbb{R}^{ m}}\left(\fint_{E}\left|\,{\bf f}(x)-{\bf h}\,\right|^{p}\operatorname{d}\!x\right)^{\frac{1}{p}}\,.
\end{equation}

\subsection{Comparison systems and estimates.}

This subsection is devoted to compare the weak solution pair of \eqref{model} to that of desired problem for which we have known regularity results.
From now on, we consider the following localized problem, homogeneous problem and desired problem, respectively.
\begin{equation}\label{localmodel}
\left\{\begin{array}{r@{\ \ }c@{\ \ }ll}
\operatorname{div}\left({\bf A}(x)D{\bf u}\right)-\nabla\pi & =& \operatorname{div}{\bf F} & \mbox{in}\ \ B_{2R}(x_{0})\,, \\[0.05cm]
\operatorname{div} {\bf u}&=& 0 & \mbox{in}\ \ B_{2R}(x_{0})\,.
\end{array}\right.
\end{equation}
\begin{equation}\label{ComparisonSystem}
\left\{\begin{array}{r@{\ \ }c@{\ \ }ll}
\operatorname{div}\left({\bf A}(x)D{\bf v}\right)-\bf\nabla\pi_{ v} & =& {\bf 0} & \mbox{in}\ \ B_{2R}(x_{0})\,, \\[0.05cm]
\operatorname{div}{\bf v}&=& 0 & \mbox{in}\ \ B_{2R}(x_{0})\,, \\[0.05cm]
{\bf v}&=&{\bf u} & \mbox{on}\ \ \partial B_{2R}(x_{0})\,,
\end{array}\right.
\end{equation}
where $\pi_{\bf v}$ is an associated pressure of $\bf v$.
\begin{equation}\label{ComparisonSystem-2}
\left\{\begin{array}{r@{\ \ }c@{\ \ }ll}
\operatorname{div}\left(\overline{{\bf A}}_{B_{2R}(x_{0})}D{\bf w}\right)-\bf\nabla\pi_{ w} & =&  {\bf 0} & \mbox{in}\ \ B_{\frac{3R}{2}}(x_{0})\,, \\[0.05cm]
\operatorname{div}{\bf w}&=& 0 & \mbox{in}\ \ B_{\frac{3R}{2}}(x_{0})\,, \\[0.05cm]
{\bf w} &=&  {\bf v} & \mbox{on}\ \ \partial B_{\frac{3R}{2}}(x_{0})\,,
\end{array}\right.
\end{equation}
where $\pi_{\bf w}$ is an associated pressure of $\bf w$. Note that the definitions of the weak solution pairs to problem \eqref{localmodel}\,, \eqref{ComparisonSystem} and \eqref{ComparisonSystem-2} can be similarly determined as Definition \ref{Def-weaksolution}\,.

Now we proceed by establishing a comparison estimate regarding to
$D{\bf u}$ and $D{\bf v}$\,, as well as the associated
pressures $\pi$ and $\pi_{\bf v}$\,.
\begin{lemma}\label{approximation1}
Let $\left({\bf u},\pi\right)$ be a weak solution pair to \eqref{localmodel} with ${\bf F}\in L^{2}(\Omega\,,\mathbb{R}^{n\times n})$\,, then one can find a weak solution pair $\left({\bf v},\pi_{\bf v}\right)$ to \eqref{ComparisonSystem} such that
\begin{eqnarray}\label{inequ-comparisonlemma1}
  \fint_{B_{2R}(x_{0})}\left|D{\bf u}-D{\bf v}\right|^{2}+\left|\pi-\pi_{\bf v}\right|^{2}\operatorname{d}\!x
\leq C\fint_{B_{2R}(x_{0})}\!\left|{\bf F}-\overline{{\bf F}}_{B_{2R}(x_{0})}\right|^{2}\operatorname{d}\!x
\end{eqnarray}
and
\begin{equation}\label{u-v}
   \fint_{B_{2R}(x_{0})}\left|{\bf u}-{\bf v}\right|^{2}\operatorname{d}\!x\leq CR^{2}\fint_{B_{2R}(x_{0})}\!\left|{\bf F}-\overline{{\bf F}}_{B_{2R}(x_{0})}\right|^{2}\operatorname{d}\!x
\end{equation}
for almost all $x_{0}\in\Omega$ and every $B_{2R}(x_{0})\subset\Omega$, where the positive constant $C=C(n,\lambda,\Lambda)$.
\end{lemma}
\begin{proof}
Let $\left({\bf u},\pi\right)$ and $\left({\bf v},\pi_{\bf v}\right)$ be the weak solution pairs to \eqref{localmodel} and \eqref{ComparisonSystem} respectively, then $({\bf u}-{\bf v}, \pi-\pi_{\bf v})\in W_{0,\sigma}^{1,2}(B_{2R}(x_{0}), \mathbb{R}^{n})\times L^{2}(B_{2R}(x_{0}))$ is a weak solution pair to
\begin{equation}\label{lemu-v-equ1}
\left\{\begin{array}{r@{\ \ }c@{\ \ }ll}
\operatorname{div}\left({\bf A}(x)\left(D{\bf u}-D{\bf v}\right)\right)-\nabla\left(\pi-\pi_{\bf v}\right) & =& \operatorname{div}\left({\bf F}-\overline{{\bf F}}_{B_{2R}(x_{0})}\right) & \mbox{in}\ \ B_{2R}(x_{0})\,, \\[0.05cm]
\operatorname{div}\left({\bf u}-{\bf v}\right)&=& 0 & \mbox{in}\ \ B_{2R}(x_{0})\,, \\[0.05cm]
{\bf u}-{\bf v}&=& {\bf 0}& \mbox{on}\ \ \partial B_{2R}(x_{0})\,.
\end{array}\right.
\end{equation}
We can select $ {\bf u}-{\bf v}$ as
a divergence free test function for \eqref{lemu-v-equ1}. Thus by virtue of the uniform ellipticity condition \eqref{uniformellipticity} and Young's inequality, we obtain
\begin{eqnarray}\label{lemu-v-inequ2}
 && \lambda\,\fint_{B_{2R}(x_{0})}\left|D{\bf u}-D{\bf  v}\right|^{2}\operatorname{d}\!x\nonumber\\
&\leq& \,\fint_{B_{2R}(x_{0})} \left\langle\,{\bf A}(x)\left(D{\bf  u}-D{\bf v}\right),D{\bf u}-D{\bf  v}\right\rangle\operatorname{d}\!x,\nonumber\\
&=& \fint_{B_{2R}(x_{0})}\left\langle{\bf F}-\overline{{\bf F}}_{B_{2R}(x_{0})}\,,D{\bf u}-D{\bf v}\right\rangle\,\operatorname{d}\!x,\nonumber\\
&\leq&\tau_{1}\fint_{B_{2R}(x_{0})}\left|D{\bf  u}-D{\bf v}\right|^{2}\operatorname{d}\!x+ C(\tau_{1})\,\fint_{B_{2R}(x_{0})}\left|{\bf F}-\overline{{\bf F}}_{B_{2R}(x_{0})}\right|^{2}\operatorname{d}\!x\,.
\end{eqnarray}
Choosing $\tau_{1}=\frac{\lambda}{2}$\,, then we derive
\begin{equation}\label{lemu-v-inequ3}
\fint_{B_{2R}(x_{0})}\left|D{\bf u}-D{\bf v}\right|^{2}\operatorname{d}\!x\leq C(\lambda)\fint_{B_{2R}(x_{0})}\!\left|{\bf F}-\overline{{\bf F}}_{B_{2R}(x_{0})}\right|^{2}\operatorname{d}\!x\,.
\end{equation}
Applying Poincar\'{e}'s inequality, we deduce \eqref{u-v}, that is
\begin{eqnarray*}
\fint_{B_{2R}(x_{0})}\left|{\bf u}-{\bf v}\right|^{2}\operatorname{d}\!x&\leq& C(n)R^{2}\fint_{B_{2R}(x_{0})}\!\left|D{\bf u}-D{\bf v}\right|^{2}\operatorname{d}\!x,\\
   &\leq& C(n,\lambda)R^{2}\fint_{B_{2R}(x_{0})}\!\left|{\bf F}-\overline{{\bf F}}_{B_{2R}(x_{0})}\right|^{2}\operatorname{d}\!x.
\end{eqnarray*}

On the other hand, let $\varphi\in W_{0}^{1,2}\left(B_{2R}(x_{0}),\mathbb{R}^{n}\right)$ be a test function of \eqref{lemu-v-equ1}\, then we have
\begin{equation}\label{lemu-v-equ4}
  \int_{B_{2R}(x_{0})}\left(\pi-\pi_{\bf v}\right)\operatorname{div}\varphi \operatorname{d}\!x= \int_{B_{2R}(x_{0})}\big\langle\left[{\bf A}(x)\left(D{\bf  u}-D{\bf v}\right)-\left({\bf F}-\overline{{\bf F}}_{B_{2R}(x_{0})}\right)\right],\nabla\varphi\big\rangle \operatorname{d}\!x.
\end{equation}
More precisely, selecting the above ${\bf \varphi}$ be a solution of the auxiliary problem
\begin{equation}\label{auxiliaryprob}
\left\{\begin{array}{r@{\ \ }c@{\ \ }ll}
\operatorname{div}{\bf \varphi} & =&\pi-\pi_{\bf v} & \mbox{in}\ \ B_{2R}(x_{0})\,, \\[0.05cm]
{\bf \varphi} &=&  0 & \mbox{on}\ \ \partial B_{2R}(x_{0})\,,
\end{array}\right.
\end{equation}
where $\pi-\pi_{\bf v}\in L^{2}(B_{2R}(x_{0}))$. By virtue of
Lemma \ref{uniqueness-Johndomain} to reveal that the associated
pressure is determined uniquely up to a constant, so we can find a weak solution pair $\left({\bf v},\pi_{\bf v}\right)$  to \eqref{ComparisonSystem} such that
$\int_{B_{2R}(x_{0})}\pi-\pi_{\bf v}\operatorname{d}\!x=0$.
Then the existence of such a solution to auxiliary problem
\eqref{auxiliaryprob} is ensured by Lemma \ref{existence-Johndomain}
and we obtain
\begin{equation}\label{lemu-v-inequ5}
  \left\|\bf\nabla\varphi\right\|_{L^{2}\left(B_{2R}(x_{0})\right)}\leq c_{1}\left\|\pi-\pi_{\bf v}\right\|_{L^{2}\left(B_{2R}(x_{0})\right)},
\end{equation}
where the positive constant $c_{1}$ depends only on $n$. Substituting such $\varphi$ into equality
\eqref{lemu-v-equ4} and combining Young's inequality with
\eqref{lemu-v-inequ5} yield that
\begin{eqnarray*}
   && \int_{B_{2R}(x_{0})}\left|\pi-\pi_{\bf v}\right|^{2} \operatorname{d}\!x\nonumber\\
   &=& \int_{B_{2R}(x_{0})}\big\langle\left[{\bf A}(x)\left(D{\bf  u}-D{\bf v}\right)-\left({\bf F}-\overline{{\bf F}}_{B_{2R}(x_{0})}\right)\right],{\bf\nabla\varphi}\big\rangle \operatorname{d}\!x, \nonumber\\
  &\leq& C(\tau_{2})\int_{B_{2R}(x_{0})}\left|{\bf A}(x)\left(D{\bf  u}-D{\bf v}\right)-\left({\bf F}-\overline{{\bf F}}_{B_{2R}(x_{0})}\right)\right|^{2}\operatorname{d}\!x+\tau_{2} \int_{B_{2R}(x_{0})}\left|\bf\nabla\varphi\right|^{2}\operatorname{d}\!x,\nonumber \\
   &\leq& C(\tau_{2})\int_{B_{2R}(x_{0})}\left|{\bf A}(x)\left(D{\bf  u}-D{\bf v}\right)-\left({\bf F}-\overline{{\bf F}}_{B_{2R}(x_{0})}\right)\right|^{2}\operatorname{d}\!x+c_{1}^{2}\tau_{2} \int_{B_{2R}(x_{0})}\left|\pi-\pi_{\bf v}\right|^{2}\operatorname{d}\!x.
\end{eqnarray*}
We can choose the positive constant $\tau_{2}$ sufficiently small
such that $c_{1}^{2}\tau_{2}\leq\frac{1}{2}$ and utilize the boundedness of ${\bf A}$ and \eqref{lemu-v-inequ3}, one infers that
\begin{eqnarray}\label{lemu-v-inequ6}
 && \fint_{B_{2R}(x_{0})}\left|\pi-\pi_{\bf v}\right|^{2} \operatorname{d}\!x\nonumber\\
  &\leq&C\fint_{B_{2R}(x_{0})}\left|D{\bf  u}-D{\bf v}\right|^{2}\operatorname{d}\!x + C\fint_{B_{2R}(x_{0})}\left|{\bf F}-\overline{{\bf F}}_{B_{2R}(x_{0})}\right|^{2}\operatorname{d}\!x,\nonumber\\
  &\leq& C\fint_{B_{2R}(x_{0})}\!\left|{\bf F}-\overline{{\bf F}}_{B_{2R}(x_{0})}\right|^{2}\operatorname{d}\!x\,,
\end{eqnarray}
where $C=C(n,\lambda,\Lambda)$. Hence, a combination of \eqref{lemu-v-inequ3} and \eqref{lemu-v-inequ6} concludes \eqref{inequ-comparisonlemma1}. The proof of Lemma \ref{approximation1} is completed.
\end{proof}

The second part of this subsection will be turned to derive the comparison estimate for ${\bf\nabla v}$ and the associated pressure
$\pi_{\bf v}$ to \eqref{ComparisonSystem} with ${\bf\nabla w}$ and
the associated pressure $\pi_{\bf w}$ to
\eqref{ComparisonSystem-2}\,.

\begin{lemma}\label{approximation2}
Let $\left( {\bf v},\pi_{\bf v}\right)$ be a weak solution pair to \eqref{ComparisonSystem}, then there exist a weak solution pair $\left({\bf w},\pi_{\bf w}\right)$ to \eqref{ComparisonSystem-2} and an exponent $q=q(n,\lambda,\Lambda)>2$ such that
\begin{equation}\label{inequ-comparisonlemma2}
\fint_{B_{\frac{3R}{2}}(x_{0})}\left|D{\bf v}-D{\bf w}\right|^{2}+\left|\pi_{\bf v}-\pi_{\bf  w}\right|^{2}\operatorname{d}\!x\leq C\left([{\bf A}]_{\operatorname{ BMO}}^{2}(2R)\right)^{1-\frac{2}{q}}\left(\fint_{B_{2R}(x_{0})}\left|D{\bf v}\right|\operatorname{d}\!x\right)^{2}\,,
\end{equation}
for almost all $x_{0}\in\Omega$ and every $B_{2R}(x_{0})\subset\Omega$, where the positive constant $C=C(n,\lambda,\Lambda)$.
\end{lemma}

\begin{proof}
A direct computation reveals that $\left( {\bf v}-{\bf w},\pi_{\bf v}-\pi_{\bf w}\right)\in W_{0,\sigma}^{1,2}(B_{\frac{3R}{2}}(x_{0}), \mathbb{R}^{n})\times L^{2}(B_{\frac{3R}{2}}(x_{0}))$ is a weak solution pair to
\begin{equation}\label{lemv-w-equ1}
\left\{\begin{array}{r@{\ \ }c@{\ \ }ll}
\operatorname{div}\left(\overline{{\bf A}}_{B_{2R}(x_{0})}\left(D{\bf v}-D{\bf w}\right)\right)-\nabla\left(\pi_{\bf v}-\pi_{\bf w}\right) & =& \operatorname{div}\left(\left(\overline{{\bf A}}_{B_{2R}(x_{0})}-{\bf A}(x)\right)D{\bf v}\right) & \mbox{in}\ \ B_{\frac{3R}{2}}(x_{0})\,, \\[0.05cm]
\operatorname{div}\left({\bf v}- {\bf w}\right)&=& 0 & \mbox{in}\ \ B_{\frac{3R}{2}}(x_{0})\,, \\[0.05cm]
{\bf v}- {\bf w} &=& {\bf 0} & \mbox{on}\ \
\partial B_{\frac{3R}{2}}(x_{0})\,.
\end{array}\right.
\end{equation}
Selecting ${\bf v}- {\bf w}$ as a divergence free test function of \eqref{lemv-w-equ1}, then a combination of the uniform ellipticity condition \eqref{uniformellipticity} and Young's inequality yields that
\begin{eqnarray}\label{lemv-w-inequ2}
 && \lambda\,\fint_{B_{\frac{3R}{2}}(x_{0})}\left|D{\bf v}-D{\bf w}\right|^{2}\operatorname{d}\!x\nonumber\\
&\leq& \,\fint_{B_{\frac{3R}{2}}(x_{0})} \left\langle\,\overline{{\bf A}}_{B_{2R}(x_{0})}\left(D{\bf v}-D{\bf w}\right),D{\bf v}-D{\bf  w}\right\rangle\operatorname{d}\!x,\nonumber\\
&=& \fint_{B_{\frac{3R}{2}}(x_{0})}\left\langle\left(\overline{{\bf A}}_{B_{2R}(x_{0})}-{\bf A}(x)\right)D{\bf v}\,,D{\bf v}-D{\bf w}\right\rangle\,\operatorname{d}\!x,\nonumber\\
&\leq&\epsilon_{1}\fint_{B_{\frac{3R}{2}}(x_{0})}\left|D{\bf v}-D{\bf w}\right|^{2}\operatorname{d}\!x+ C(\epsilon_{1})\,\fint_{B_{\frac{3R}{2}}(x_{0})}\left|\overline{{\bf A}}_{B_{2R}(x_{0})}-{\bf A}(x)\right|^{2}\left|D{\bf v}\right|^{2}\operatorname{d}\!x\,,
\end{eqnarray}
where we select $\epsilon_{1}=\frac{\lambda}{2}$. By virtue of \cite[Theorem 2.2]{GMo}, there exists an exponent $q=q(\lambda,\Lambda,n)>2$ such that the following reverse H\"{o}lder type inequality
\begin{equation}\label{v-reverseholder}
  \fint_{B_{\frac{3R}{2}}(x_{0})}|D{\bf v}|^{q}dx\leq C\left(\fint_{B_{2R}(x_{0})}|D{\bf v}|^{2}dx\right)^{\frac{q}{2}}
\end{equation}
holds. Then it follows from Lemma \ref{reverse-holder} that
\begin{equation}\label{reverse-Dv}
  \fint_{B_{\frac{3R}{2}}(x_{0})}|D{\bf v}|^{q}dx\leq C\left(\fint_{B_{2R}(x_{0})}|D{\bf v}|dx\right)^{q},
\end{equation}
where $C=C(n,\lambda,\Lambda)$. Applying H\"{o}lder's inequality, \eqref{reverse-Dv}, the boundedness of $\bf A$ and \eqref{smallBMO} to last term on the right side of \eqref{lemv-w-inequ2}, we derive
\begin{eqnarray}\label{lemv-w-inequ3}
 && \fint_{B_{\frac{3R}{2}}(x_{0})}\left|D{\bf v}-D{\bf w}\right|^{2}\operatorname{d}\!x \nonumber \\
  &\leq& C(\lambda)\,\left(\fint_{B_{\frac{3R}{2}}(x_{0})}\left|D{\bf v}\right|^{q}\operatorname{d}\!x\right)^{\frac{2}{q}}\left(\fint_{B_{\frac{3R}{2}}(x_{0})}\left|\overline{{\bf A}}_{B_{2R}(x_{0})}-{\bf A}(x)\right|^{\frac{2q}{q-2}}\operatorname{d}\!x\right)^{\frac{q-2}{q}}\,, \nonumber\\
   &\leq&C(n,\lambda,\Lambda)\,\left(\fint_{B_{2R}(x_{0})}\left|D{\bf v}\right|\operatorname{d}\!x\right)^{2}\left(\fint_{B_{2R}(x_{0})}\left|\overline{{\bf A}}_{B_{2R}(x_{0})}-{\bf A}(x)\right|^{2+\frac{4}{q-2}}\operatorname{d}\!x\right)^{\frac{q-2}{q}}\,, \nonumber\\
 &\leq&C(n,\lambda,\Lambda)\,\left(\fint_{B_{2R}(x_{0})}\left|D{\bf v}\right|\operatorname{d}\!x\right)^{2}\left(\fint_{B_{2R}(x_{0})}\left|\overline{{\bf A}}_{B_{2R}(x_{0})}-{\bf A}(x)\right|^{2}\operatorname{d}\!x\right)^{\frac{q-2}{q}}\,, \nonumber\\
 &\leq&C(n,\lambda,\Lambda)\left([{\bf A}]_{\operatorname{ BMO}}^{2}(2R)\right)^{1-\frac{2}{q}}\left(\fint_{B_{2R}(x_{0})}\left|D{\bf v}\right|\operatorname{d}\!x\right)^{2}\,.
\end{eqnarray}

Furthermore, it remains to be estimated $\fint_{B_{\frac{3R}{2}}(x_{0})}\left|\pi_{\bf v}-\pi_{\bf w}\right|^{2}\operatorname{d}\!x$. Choosing $\phi\in W_{0}^{1,2}\left(B_{\frac{3R}{2}}(x_{0}),\mathbb{R}^{n}\right)$ as a test function of \eqref{lemv-w-equ1}, that is to say
\begin{equation}\label{lemv-w-equ4}
  \fint_{B_{\frac{3R}{2}}(x_{0})}\left(\pi_{\bf v}-\pi_{\bf w}\right)\operatorname{div}\phi \operatorname{d}\!x= \fint_{B_{\frac{3R}{2}}(x_{0})}\big\langle\left[\overline{{\bf A}}_{B_{2R}(x_{0})}\left(D{\bf v}-D{\bf w}\right)-\left(\overline{{\bf A}}_{B_{2R}(x_{0})}-{\bf A}(x)\right)D{\bf v}\right],\nabla\phi\big\rangle \operatorname{d}\!x.
\end{equation}
More precisely, selecting the above $\phi$ be a solution of the auxiliary problem
\begin{equation}\label{auxiliaryprob2}
\left\{\begin{array}{r@{\ \ }c@{\ \ }ll}
\operatorname{div}{\bf \phi} & =&\pi_{\bf v}-\pi_{\bf w} & \mbox{in}\ \ B_{\frac{3R}{2}}(x_{0})\,, \\[0.05cm]
{\bf \phi} &=&  0 & \mbox{on}\ \ \partial B_{\frac{3R}{2}}(x_{0})\,,
\end{array}\right.
\end{equation}
where $\pi_{\bf v}-\pi_{\bf w}\in L^{2}(B_{\frac{3R}{2}}(x_{0}))$.
Lemma \ref{uniqueness-Johndomain} indicates that the associated
pressure is determined uniquely up to a constant, so we can find a weak solution pair $\left({\bf w},\pi_{\bf w}\right)$  to \eqref{ComparisonSystem-2} such that
$\int_{B_{\frac{3R}{2}}(x_{0})}\pi_{\bf v}-\pi_{\bf w}\operatorname{d}\!x=0$.
Then the existence of such a solution to auxiliary problem
\eqref{auxiliaryprob2} is ensured by Lemma \ref{existence-Johndomain} and we have
\begin{equation}\label{lemv-w-inequ5}
  \left\|\bf\nabla\phi\right\|_{L^{2}\left(B_{\frac{3R}{2}}(x_{0})\right)}\leq c_{2}\left\|\pi_{\bf v}-\pi_{\bf w}\right\|_{L^{2}\left(B_{\frac{3R}{2}}(x_{0})\right)},
\end{equation}
where the positive constant $c_{2}$ depends only on $n$. Substituting such $\phi$ into equality
\eqref{lemv-w-equ4} and combining Young's inequality with
\eqref{lemv-w-inequ3}, \eqref{lemv-w-inequ5} and the boundedness of ${\bf A}$ deduce that
\begin{eqnarray*}
   && \fint_{B_{\frac{3R}{2}}(x_{0})}\left|\pi_{\bf v}-\pi_{\bf w}\right|^{2} \operatorname{d}\!x\nonumber\\
  &\leq& C(\epsilon_{2})\fint_{B_{\frac{3R}{2}}(x_{0})}\left|\overline{{\bf A}}_{B_{2R}(x_{0})}\left(D{\bf v}-D{\bf w}\right)-\left(\overline{{\bf A}}_{B_{2R}(x_{0})}-{\bf A}(x)\right)D{\bf v}\right|^{2}\operatorname{d}\!x+\epsilon_{2} \fint_{B_{\frac{3R}{2}}(x_{0})}\left|\bf\nabla\phi\right|^{2}\operatorname{d}\!x,\\
   &\leq& C(\epsilon_{2},n,\Lambda)\fint_{B_{\frac{3R}{2}}(x_{0})}\left|D{\bf v}-D{\bf w}\right|^{2}\operatorname{d}\!x+ C(\epsilon_{2},n,\lambda,\Lambda)\left([{\bf A}]_{\operatorname{ BMO}}^{2}(2R)\right)^{1-\frac{2}{q}}\left(\fint_{B_{\frac{3R}{2}}(x_{0})}\left|D{\bf v}\right|\operatorname{d}\!x\right)^{2}\\
   &&+c_{2}^{2}\epsilon_{2} \fint_{B_{\frac{3R}{2}}(x_{0})}\left|\pi_{\bf v}-\pi_{\bf w}\right|^{2}\operatorname{d}\!x,\\
    &\leq& C(\epsilon_{2},n,\lambda,\Lambda)\left([{\bf A}]_{\operatorname{ BMO}}^{2}(2R)\right)^{1-\frac{2}{q}}\left(\fint_{B_{\frac{3R}{2}}(x_{0})}\left|D{\bf v}\right|\operatorname{d}\!x\right)^{2}+c_{2}^{2}\epsilon_{2} \fint_{B_{\frac{3R}{2}}(x_{0})}\left|\pi_{\bf v}-\pi_{\bf w}\right|^{2}\operatorname{d}\!x.
\end{eqnarray*}
Finally, choosing the positive constant $\epsilon_{2}$ sufficiently small
such that $c_{2}^{2}\epsilon_{2}\leq\frac{1}{2}$, we conclude that
\begin{equation}\label{lemv-w-inequ6}
 \fint_{B_{\frac{3R}{2}}(x_{0})}\left|\pi_{\bf v}-\pi_{\bf w}\right|^{2} \operatorname{d}\!x
  \leq C\left([{\bf A}]_{\operatorname{ BMO}}^{2}(2R)\right)^{1-\frac{2}{q}}\left(\fint_{B_{\frac{3R}{2}}(x_{0})}\left|D{\bf v}\right|\operatorname{d}\!x\right)^{2},
\end{equation}
where $C=C(n,\lambda,\Lambda)$. Hence, a combination of \eqref{lemv-w-inequ3} and \eqref{lemv-w-inequ6} yields \eqref{inequ-comparisonlemma2}. The proof of Lemma \ref{approximation2} is completed.
\end{proof}

Ultimately, by virtue of Lemma \ref{approximation1}\,, Lemma \ref{approximation2}\, with H\"{o}lder's inequality, we can directly obtain the comparison with $D{\bf u}$, $\pi$ and $D{\bf w}$, $\pi_{\bf w}$ as follows.
\begin{lemma}\label{Lemma-comparison}
Let $\left( {\bf u},\pi\right)$ be a weak solution pair to \eqref{localmodel}, then there exist a weak solution pair $\left({\bf w},\pi_{\bf w}\right)$ to \eqref{ComparisonSystem-2} and an exponent $q=q(n,\lambda,\Lambda)>2$ such that
\begin{eqnarray}\label{inequ-comparison}
\fint_{B_{\frac{3R}{2}}(x_{0})}\left|D{\bf u}-D{\bf w}\right|^{2}+\left|\pi-\pi_{\bf  w}\right|^{2}\operatorname{d}\!x
&\leq& C\left([{\bf A}]_{\operatorname{ BMO}}^{2}(2R)\right)^{1-\frac{2}{q}}\left(\fint_{B_{2R}(x_{0})}\left|D{\bf u}\right|\operatorname{d}\!x\right)^{2}\nonumber\\
&&+C\left(1+\left([{\bf A}]_{\operatorname{ BMO}}^{2}(2R)\right)^{1-\frac{2}{q}}\right)\fint_{B_{2R}(x_{0})}\!\left|{\bf F}-\overline{{\bf F}}_{B_{2R}(x_{0})}\right|^{2}\operatorname{d}\!x\nonumber\\
\end{eqnarray}
for almost all $x_{0}\in\Omega$ and every $B_{2R}(x_{0})\subset\Omega$, where the positive constant $C=C(n,\lambda,\Lambda)$.
\end{lemma}

\section{Proof of main theorems.}\label{section3}

In this section, we first prove the pointwise gradient estimate in Theorem \ref{Th1}\,. Before proceeding further, we introduce the following Campanato-type decay estimate for the gradient of solution to \eqref{ComparisonSystem-2} (cf. \cite{DiKaSc, Dong})\,.
\begin{lemma}\label{lem-w-decayestimate}
Let ${w}$ be the weak solution to \eqref{ComparisonSystem-2}\,, then there exist positive constants $\alpha=\alpha(n,\lambda,\Lambda)$ and $C=C(n,\lambda,\Lambda)$ such that
\begin{equation}\label{w-decayestimate}
 \fint_{B_{\sigma R}}\left|D{\bf w}-\overline{D{\bf w}}_{B_{\sigma R}}\right|^{2}\operatorname{d}\!x \leq C\,\sigma^{2\alpha} \fint_{B_{R}}\left|D{\bf w}-\overline{D{\bf w}}_{B_{R}}\right|^{2}\operatorname{d}\!x
\end{equation}
for every $0<\sigma\leq1$ and every concentric ball $B_{\sigma R}\subset B_{R}\subset\Omega$.
\end{lemma}
We now turn our attention to the Campanato-type decay estimate of $D{\bf u}$, which is the main ingredient to carry on the proof of Theorem \ref{Th1}\,.
\begin{lemma}\label{lem-u-decayestimate}
Let $\beta\in(0,1)$ and $\left({\bf u},\pi\right)$ be a weak
solution pair to \eqref{model} with ${\bf F}\in
L^{2}(\Omega,\mathbb{R}^{n\times n})$. There exist the positive
constants $\sigma=\sigma(n,\lambda,\Lambda,\beta)\in\left(0,\frac{1}{2}\right]$, $\delta=\delta(n,\lambda,\Lambda)$ and $R_{1}=R_{1}(n,\lambda,\Lambda)$
such that if ${\bf A}$ is $(\delta,R)$\,-vanishing for some $R>0$, then
\begin{eqnarray}\label{u-decayestimate}
&& \fint_{B_{\sigma\!R}(x_{0})}\left|D{\bf
u}-\overline{D{\bf
u}}_{B_{\sigma\!R}(x_{0})}\right|+\left|{\pi}-\overline{\pi}_{B_{\sigma\!R}(x_{0})}\right|\operatorname{d}\!x\nonumber\\
&\leq &\beta\fint_{B_{R}(x_{0})}\left|D{\bf
u}-\overline{D {\bf
u}}_{B_{R}(x_{0})}\right|\operatorname{d}\!x
 + C_{\beta}\left(\fint_{B_{R}(x_{0})}\left|{\bf
F}-\overline{\bf
F}_{B_{R}(x_{0})}\right|^{2}\operatorname{d}\!x\right)^{\frac{1}{2}}\nonumber\\
&&+C_{\beta}\left([{\bf A}]_{\operatorname{ BMO}}^{2}(R)\right)^{\frac{1}{2}-\frac{1}{q}}\fint_{B_{R}(x_{0})}\left|D{\bf
u}\right|\operatorname{d}\!x
\end{eqnarray}
for any $B_{R}(x_{0})\subset\Omega$ with $0<R\leq R_{1}$, where $C_{\beta}=C_{\beta}(n,\lambda,\Lambda,\beta)>0$\,.
\end{lemma}

\begin{proof}
In order to prove this technical lemma, we first utilize H\"{o}lder's inequality and \eqref{eqn-minimal2}, one obtains that
\begin{eqnarray}\label{lem-u-decay-inequ1}
&& \left(\fint_{B_{\sigma\!R}(x_{0})}\left|D{\bf
u}-\overline{D{\bf
u}}_{B_{\sigma\!R}(x_{0})}\right|+\left|{\pi}-\overline{\pi}_{B_{\sigma\!R}(x_{0})}\right|\operatorname{d}\!x\right)^{2}\nonumber\\
&\leq&2\fint_{B_{\sigma\!R}(x_{0})}\left|D{\bf
u}-\overline{D{\bf
u}}_{B_{\sigma\!R}(x_{0})}\right|^{2}\operatorname{d}\!x+2\fint_{B_{\sigma\!R}(x_{0})}\left|{\pi}-\overline{\pi}_{B_{\sigma\!R}(x_{0})}\right|
^{2}\operatorname{d}\!x,\nonumber\\
&\leq&4\fint_{B_{\sigma\!R}(x_{0})}\left|D{\bf
u}-D{\bf
w}\right|^{2}+\left|D{\bf
w}-\overline{D{\bf
w}}_{B_{\sigma\!R}(x_{0})}\right|^{2}\operatorname{d}\!x+4\fint_{B_{\sigma\!R}(x_{0})}\left|{\pi}-\pi_{\bf w}\right|^{2}+
\left|\pi_{\bf w}-\overline{\pi_{\bf w}}_{B_{\sigma\!R}(x_{0})}\right|^{2}\operatorname{d}\!x,\nonumber\\
&\leq&C(n)\sigma^{-n}\fint_{B_{\frac{3R}{2}}(x_{0})}\left|D{\bf
u}-D{\bf
w}\right|^{2}+\left|{\pi}-\pi_{\bf w}\right|^{2}\operatorname{d}\!x\nonumber\\
&&+4\fint_{B_{\sigma\!R}(x_{0})}\left|D{\bf
w}-\overline{D{\bf
w}}_{B_{\sigma\!R}(x_{0})}\right|^{2}+
\left|\pi_{\bf w}-\overline{\pi_{\bf w}}_{B_{\sigma\!R}(x_{0})}\right|^{2}\operatorname{d}\!x
\end{eqnarray}
for any $0<\sigma\leq1$. To estimate the last term on the right side
of \eqref{lem-u-decay-inequ1}\,, let $\psi\in W_{0}^{1,2}\left(B_{\sigma\!R}(x_{0}),\mathbb{R}^{n}\right)$ be a test function of \eqref{ComparisonSystem-2}, i.e.,
\begin{equation}\label{lem-u-decay-equ2}
  \fint_{B_{\sigma\!R}(x_{0})}\left(\pi_{\bf w}-\overline{\pi_{\bf w}}_{B_{\sigma\!R}(x_{0})}\right)\operatorname{div}\psi \operatorname{d}\!x= \fint_{B_{\sigma\!R}(x_{0})}\big\langle\overline{{\bf A}}_{B_{2R}(x_{0})}\left(D{\bf w}-\overline{D{\bf
w}}_{B_{\sigma\!R}(x_{0})}\right),\nabla\psi\big\rangle \operatorname{d}\!x.
\end{equation}
More precisely, selecting the above $\psi$ be a solution of the auxiliary problem
\begin{equation}\label{auxiliaryprob3}
\left\{\begin{array}{r@{\ \ }c@{\ \ }ll}
\operatorname{div}{\bf \psi} & =&\pi_{\bf w}-\overline{\pi_{\bf w}}_{B_{\sigma\!R}(x_{0})} & \mbox{in}\ \ B_{\sigma\!R}(x_{0})\,, \\[0.05cm]
{\bf \psi} &=&  0 & \mbox{on}\ \ \partial B_{\sigma\!R}(x_{0})\,,
\end{array}\right.
\end{equation}
where $\pi_{\bf w}-\overline{\pi_{\bf w}}_{B_{\sigma\!R}(x_{0})}\in L^{2}(B_{\sigma\!R}(x_{0}))$ and $\int_{B_{\sigma\!R}(x_{0})}\pi_{\bf w}-\overline{\pi_{\bf w}}_{B_{\sigma\!R}(x_{0})}\operatorname{d}\!x=0$. Then Lemma \ref{existence-Johndomain} infers there exists a
 solution to auxiliary problem
\eqref{auxiliaryprob3} such that
\begin{equation}\label{lem-u-decay-inequ3}
  \left\|\bf\nabla\psi\right\|_{L^{2}\left(B_{\sigma\!R}(x_{0})\right)}\leq c_{3}\left\|\pi_{\bf w}-\overline{\pi_{\bf w}}_{B_{\sigma\!R}(x_{0})}\right\|_{L^{2}\left(B_{\sigma\!R}(x_{0})\right)},
\end{equation}
where the positive constant $c_{3}$ depends only on $n$. Substituting such $\psi$ into equality
\eqref{lem-u-decay-equ2} and combining Young's inequality with
\eqref{lem-u-decay-inequ3} and the boundedness of ${\bf A}$ yield that
\begin{eqnarray*}
   && \fint_{B_{\sigma\!R}(x_{0})}\left|\pi_{\bf w}-\overline{\pi_{\bf w}}_{B_{\sigma\!R}(x_{0})}\right|^{2} \operatorname{d}\!x\nonumber\\
  &\leq& C(\varepsilon,n,\Lambda)\fint_{B_{\sigma\!R}(x_{0})}\left|D{\bf
w}-\overline{D{\bf w}}_{B_{\sigma\!R}(x_{0})}\right|^{2}\operatorname{d}\!x+\varepsilon \fint_{B_{\sigma\!R}(x_{0})}\left|\nabla\psi\right|^{2}\operatorname{d}\!x,\\
   &\leq& C(\varepsilon,n,\Lambda)\fint_{B_{\sigma\!R}(x_{0})}\left|D{\bf
w}-\overline{D{\bf w}}_{B_{\sigma\!R}(x_{0})}\right|^{2}\operatorname{d}\!x+c_{3}^{2}\varepsilon \fint_{B_{\sigma\!R}(x_{0})}\left|\pi_{\bf w}-\overline{\pi_{\bf w}}_{B_{\sigma\!R}(x_{0})}\right|^{2} \operatorname{d}\!x.
\end{eqnarray*}
Setting the positive constant $\varepsilon$ sufficiently small
such that $c_{3}^{2}\varepsilon\leq\frac{1}{2}$, we deduce that
\begin{equation}\label{lem-u-decay-inequ4}
   \fint_{B_{\sigma\!R}(x_{0})}\left|\pi_{\bf w}-\overline{\pi_{\bf w}}_{B_{\sigma\!R}(x_{0})}\right|^{2} \operatorname{d}\!x\leq C(n,\Lambda)\fint_{B_{\sigma\!R}(x_{0})}\left|D{\bf
w}-\overline{D{\bf w}}_{B_{\sigma\!R}(x_{0})}\right|^{2}\operatorname{d}\!x.
\end{equation}
Inserting \eqref{lem-u-decay-inequ4} into \eqref{lem-u-decay-inequ1} and utilizing Lemma \ref{lem-w-decayestimate}\,, we derive
\begin{eqnarray}\label{lem-u-decay-inequ5}
 && \left(\fint_{B_{\sigma\!R}(x_{0})}\left|D{\bf
u}-\overline{D{\bf
u}}_{B_{\sigma\!R}(x_{0})}\right|+\left|{\pi}-\overline{\pi}_{B_{\sigma\!R}(x_{0})}\right|\operatorname{d}\!x\right)^{2}\nonumber\\
 &\leq&C\sigma^{-n}\fint_{B_{\frac{3R}{2}}(x_{0})}\left|D{\bf
u}-D{\bf
w}\right|^{2}+\left|{\pi}-\pi_{\bf w}\right|^{2}\operatorname{d}\!x+C\fint_{B_{\sigma\!R}(x_{0})}\left|D{\bf
w}-\overline{D{\bf
w}}_{B_{\sigma\!R}(x_{0})}\right|^{2}\operatorname{d}\!x,\nonumber\\
&\leq&C\sigma^{-n}\fint_{B_{\frac{3R}{2}}(x_{0})}\left|D{\bf
u}-D{\bf
w}\right|^{2}+\left|{\pi}-\pi_{\bf w}\right|^{2}\operatorname{d}\!x+C\sigma^{2\alpha}\fint_{B_{R}(x_{0})}\left|D{\bf
w}-\overline{D{\bf
w}}_{B_{R}(x_{0})}\right|^{2}\operatorname{d}\!x.
\end{eqnarray}

Next, estimating the last term on the right side of \eqref{lem-u-decay-inequ5}, we apply the following Caccioppoli inequality
\begin{equation}\label{caccioppoli-Dw}
  \fint_{B_{R}(x_{0})}\left|D^{2}{\bf w}\right|^{2}\operatorname{d}\!x\leq\frac{C(\lambda,\Lambda)}{R^{2}}\fint_{B_{\frac{3R}{2}}(x_{0})}\left|D{\bf
w}-\overline{D{\bf
w}}_{B_{\frac{3R}{2}}(x_{0})}\right|^{2}dx,
\end{equation}
which can refer to \cite{Dong}. By virtue of Sobolev embedding Theorem and \eqref{caccioppoli-Dw}, we have
\begin{equation*}
  \left(\fint_{B_{R}(x_{0})}\left|D{\bf
w}-\overline{D{\bf
w}}_{B_{R}(x_{0})}\right|^{s}\operatorname{d}\!x\right)^{\frac{1}{s}}\leq C(n,\lambda,\Lambda)\left(\fint_{B_{\frac{3R}{2}}(x_{0})}\left|D{\bf
w}-\overline{D{\bf
w}}_{B_{\frac{3R}{2}}(x_{0})}\right|^{2}\operatorname{d}\!x\right)^{\frac{1}{2}},
\end{equation*}
where
\begin{equation*}
s=
  \left\{\begin{array}{r@{\ \ }c@{\ \ }ll}
\frac{2n}{n-2},  \ & n&>2, \\
p,  \ & n&=2,
\end{array}\right.
\end{equation*}
for any $p\in[2,+\infty)$. Thus, a combination of Lemma \ref{reverse-holder} and H\"{o}lder's inequality yields that
\begin{equation}\label{lem-u-decay-inequ6}
  \left(\fint_{B_{R}(x_{0})}\left|D{\bf
w}-\overline{D{\bf
w}}_{B_{R}(x_{0})}\right|^{2}\operatorname{d}\!x\right)^{\frac{1}{2}}\leq C(n,\lambda,\Lambda)\fint_{B_{\frac{3R}{2}}(x_{0})}\left|D{\bf
w}-\overline{D{\bf
w}}_{B_{\frac{3R}{2}}(x_{0})}\right|\operatorname{d}\!x.
\end{equation}
Inserting \eqref{lem-u-decay-inequ6} into \eqref{lem-u-decay-inequ5}, then it follows from H\"{o}lder's inequality, \eqref{eqn-minimal2} and Lemma \ref{Lemma-comparison} that
\begin{eqnarray}\label{lem-u-decay-inequ7}
 && \left(\fint_{B_{\sigma\!R}(x_{0})}\left|D{\bf
u}-\overline{D{\bf
u}}_{B_{\sigma\!R}(x_{0})}\right|+\left|{\pi}-\overline{\pi}_{B_{\sigma\!R}(x_{0})}\right|\operatorname{d}\!x\right)^{2}\nonumber\\
&\leq&C\sigma^{-n}\fint_{B_{\frac{3R}{2}}(x_{0})}\left|D{\bf
u}-D{\bf
w}\right|^{2}+\left|{\pi}-\pi_{\bf w}\right|^{2}\operatorname{d}\!x+C\sigma^{2\alpha}\left(\fint_{B_{\frac{3R}{2}}(x_{0})}\left|D{\bf
w}-\overline{D{\bf
w}}_{B_{\frac{3R}{2}}(x_{0})}\right|\operatorname{d}\!x\right)^{2},\nonumber\\
  &\leq&C\left(\sigma^{-n}+\sigma^{2\alpha}\right)\fint_{B_{\frac{3R}{2}}(x_{0})}\left|D{\bf
u}-D{\bf
w}\right|^{2}+\left|{\pi}-\pi_{\bf w}\right|^{2}\operatorname{d}\!x\nonumber\\
&&+C\sigma^{2\alpha}\left(\fint_{B_{\frac{3R}{2}}(x_{0})}\left|D{\bf
u}-\overline{D{\bf
u}}_{B_{\frac{3R}{2}}(x_{0})}\right|\operatorname{d}\!x\right)^{2},\nonumber\\
   &\leq& C\sigma^{2\alpha}\left(\fint_{B_{2R}(x_{0})}\left|D{\bf
u}-\overline{D{\bf
u}}_{B_{2R}(x_{0})}\right|\operatorname{d}\!x\right)^{2}+C_{\sigma}\left([{\bf A}]_{\operatorname{ BMO}}^{2}(2R)\right)^{1-\frac{2}{q}}\left(\fint_{B_{2R}(x_{0})}\left|D{\bf u}\right|\operatorname{d}\!x\right)^{2}\nonumber\\
&&+C_{\sigma}\left(1+\left([{\bf A}]_{\operatorname{ BMO}}^{2}(2R)\right)^{1-\frac{2}{q}}\right)\fint_{B_{2R}(x_{0})}\!\left|{\bf F}-\overline{{\bf F}}_{B_{2R}(x_{0})}\right|^{2}\operatorname{d}\!x\,,
\end{eqnarray}
for any $0<\sigma\leq1$, where the positive constants $C=C(n,\lambda,\Lambda)$ and $C_{\sigma}=C(\sigma,n,\lambda,\Lambda)$.
Obviously, \eqref{lem-u-decay-inequ7} is equivalent to
\begin{eqnarray*}
&&\fint_{B_{\sigma\!R}(x_{0})}\left|D{\bf
u}-\overline{D{\bf
u}}_{B_{\sigma\!R}(x_{0})}\right|+\left|{\pi}-\overline{\pi}_{B_{\sigma\!R}(x_{0})}\right|\operatorname{d}\!x \\
 &\leq& C\sigma^{\alpha}\fint_{B_{R}(x_{0})}\left|D{\bf
u}-\overline{D{\bf
u}}_{B_{R}(x_{0})}\right|\operatorname{d}\!x+C_{\sigma} \left([{\bf A}]_{\operatorname{ BMO}}^{2}(R)\right)^{\frac{1}{2}-\frac{1}{q}}\fint_{B_{R}(x_{0})}\left|D{\bf u}\right|\operatorname{d}\!x\nonumber\\
&&+C_{\sigma}\left(1+\left([{\bf A}]_{\operatorname{ BMO}}^{2}(R)\right)^{\frac{1}{2}-\frac{1}{q}}\right)\left(\fint_{B_{R}(x_{0})}\!\left|{\bf F}-\overline{{\bf F}}_{B_{R}(x_{0})}\right|^{2}\operatorname{d}\!x\right)^{\frac{1}{2}}
\end{eqnarray*}
for any $0<\sigma\leq\frac{1}{2}$. Due to the $(\delta,R)$\,-vanishing property of ${\bf A}$, the non-decreasing function $[{\bf A}]_{\operatorname{ BMO}}^{2}(\cdot)$ and $q=q(n,\lambda,\Lambda)>2$, then there exist the positive constants $\delta=\delta(n,\lambda,\Lambda)$ and $R_{1}=R_{1}(n,\lambda,\Lambda)$ such that
$$\left([{\bf A}]_{\operatorname{ BMO}}^{2}(R)\right)^{\frac{1}{2}-\frac{1}{q}}\leq\delta^{1-\frac{2}{q}}\leq1$$
for any $R\leq R_{1}$. Finally, selecting $\sigma$ small enough such that $C\sigma^{\alpha}\leq\beta$ for any $\beta\in(0,1)$, we can conclude the desired decay estimate \eqref{u-decayestimate} for any $B_{R}(x_{0})\subset\Omega$ with $0<R\leq R_{1}$.
\end{proof}

Now we give the proof of Theorem \ref{Th1}\,.

\begin{proof}[Proof of Theorem \ref{Th1}] Without loss of generality, we may assume that
\begin{equation*}
  \int_{0}^{2R}\left( \fint_{B_{\varrho}(x_{0})}\left(\frac{|{\bf F}-\overline{{\bf F}}_{B_{\varrho}(x_{0})}|}{\varrho}\right)^{2}\operatorname{d}\!x\right)^{\frac{1}{2}}\operatorname{d}\!\varrho<\infty,
\end{equation*}
otherwise \eqref{gradient-estimate} is obviously established. We first take $R_{0}\leq R_{1}$, where the radius $R_{1}$ has been determined in Lemma \ref{lem-u-decayestimate}\,. Meanwhile, fixing $\beta=\frac{1}{4}$ and selecting the corresponding $\sigma=\sigma(n,\lambda,\Lambda)\in\left(0,\frac{1}{2}\right]$, which are given by Lemma \ref{lem-u-decayestimate}\,. After a direct calculation, we have
\begin{eqnarray}\label{th1-inequ1}
  \left|\fint_{B_{\sigma^{k}R}(x_{0})}D{\bf u}\operatorname{d}\!x-\fint_{B_{R}(x_{0})}D{\bf u}\operatorname{d}\!x \right| &=& \left|\sum_{i=0}^{k-1}\left(\fint_{B_{\sigma^{i+1}R}(x_{0})}D{\bf u}\operatorname{d}\!x-\fint_{B_{\sigma^{i}R}(x_{0})}D{\bf u}\operatorname{d}\!x\right) \right|,\nonumber \\
  &\leq& \sigma^{-n}\sum_{i=0}^{k-1}\left|\fint_{B_{\sigma^{i}R}(x_{0})}D{\bf u}-\overline{D{\bf u}}_{B_{\sigma^{i}R}(x_{0})}\operatorname{d}\!x \right|
\end{eqnarray}
for any $k\in\mathbb{N}$. Similarly,
\begin{equation}\label{th1-inequ2}
  \left|\fint_{B_{\sigma^{k}R}(x_{0})}\pi\operatorname{d}\!x-\fint_{B_{R}(x_{0})}\pi\operatorname{d}\!x \right| \leq \sigma^{-n}\sum_{i=0}^{k-1}\left|\fint_{B_{\sigma^{i}R}(x_{0})}\pi-\overline{\pi}_{B_{\sigma^{i}R}(x_{0})}\operatorname{d}\!x \right|.
\end{equation}
In terms of \eqref{u-decayestimate}, we deduce
\begin{eqnarray}\label{th1-inequ3}
  && \sum_{i=1}^{k}\left(\fint_{B_{\sigma^{i}\!R}(x_{0})}\left|D{\bf
u}-\overline{D{\bf
u}}_{B_{\sigma^{i}\!R}(x_{0})}\right|+\left|{\pi}-\overline{\pi}_{B_{\sigma^{i}\!R}(x_{0})}\right|\operatorname{d}\!x\right)\nonumber\\
&\leq &\frac{1}{4}\sum_{i=0}^{k-1}\fint_{B_{\sigma^{i}R}(x_{0})}\left|D{\bf
u}-\overline{D {\bf
u}}_{B_{\sigma^{i}R}(x_{0})}\right|\operatorname{d}\!x
 + C\sum_{i=0}^{k-1}\left(\fint_{B_{\sigma^{i}R}(x_{0})}\left|{\bf
F}-\overline{\bf
F}_{B_{\sigma^{i}R}(x_{0})}\right|^{2}\operatorname{d}\!x\right)^{\frac{1}{2}}\nonumber\\
&&+C\sum_{i=0}^{k-1}\left([{\bf A}]_{\operatorname{ BMO}}^{2}(\sigma^{i}R)\right)^{\frac{1}{2}-\frac{1}{q}}\fint_{B_{\sigma^{i}R}(x_{0})}\left|D{\bf
u}\right|\operatorname{d}\!x,\nonumber\\
&\leq &\frac{1}{4}\sum_{i=0}^{k-1}\fint_{B_{\sigma^{i}R}(x_{0})}\left|D{\bf
u}-\overline{D {\bf
u}}_{B_{\sigma^{i}R}(x_{0})}\right|\operatorname{d}\!x+C\sum_{i=0}^{k-1}\left([{\bf A}]_{\operatorname{ BMO}}^{2}(\sigma^{i}R)\right)^{\frac{1}{2}-\frac{1}{q}}\fint_{B_{\sigma^{i}R}(x_{0})}\left|D{\bf
u}-\overline{D {\bf
u}}_{B_{\sigma^{i}R}(x_{0})}\right|\operatorname{d}\!x\nonumber\\
&& + C\sum_{i=0}^{k-1}\left(\fint_{B_{\sigma^{i}R}(x_{0})}\left|{\bf
F}-\overline{\bf
F}_{B_{\sigma^{i}R}(x_{0})}\right|^{2}\operatorname{d}\!x\right)^{\frac{1}{2}}+C\sum_{i=0}^{k-1}\left([{\bf A}]_{\operatorname{ BMO}}^{2}(\sigma^{i}R)\right)^{\frac{1}{2}-\frac{1}{q}}\left|\fint_{B_{\sigma^{i}R}(x_{0})}D{\bf
u}\operatorname{d}\!x\right|,
\end{eqnarray}
where $C=C(n,\lambda,\Lambda)$. It follows from \eqref{eqn-minimal2} to estimate the integral term involving $\bf F$ of \eqref{th1-inequ3} that
\begin{eqnarray*}
  && \sum_{i=0}^{k-1}\left(\fint_{B_{\sigma^{i}R}(x_{0})}\left|{\bf
F}-\overline{\bf
F}_{B_{\sigma^{i}R}(x_{0})}\right|^{2}\operatorname{d}\!x\right)^{\frac{1}{2}}\nonumber \\
   &\leq&  \left(\fint_{B_{R}(x_{0})}\left|{\bf
F}-\overline{\bf
F}_{B_{R}(x_{0})}\right|^{2}\operatorname{d}\!x\right)^{\frac{1}{2}}+ \sum_{i=1}^{\infty}\left(\fint_{B_{\sigma^{i}R}(x_{0})}\left|{\bf
F}-\overline{\bf
F}_{B_{\sigma^{i}R}(x_{0})}\right|^{2}\operatorname{d}\!x\right)^{\frac{1}{2}}, \nonumber\\
   &=&\frac{1}{\ln2}\int_{R}^{2R}\left(\fint_{B_{R}(x_{0})}\left|{\bf
F}-\overline{\bf
F}_{B_{R}(x_{0})}\right|^{2}\operatorname{d}\!x\right)^{\frac{1}{2}}\frac{\operatorname{d}\!\varrho}{\varrho}\nonumber\\
&&+\frac{1}{\ln\frac{1}{\sigma}}
\sum_{i=1}^{\infty}\int_{\sigma^{i}R}^{\sigma^{i-1}R}\left(\fint_{B_{\sigma^{i}R}(x_{0})}\left|{\bf
F}-\overline{\bf
F}_{B_{\sigma^{i}R}(x_{0})}\right|^{2}\operatorname{d}\!x\right)^{\frac{1}{2}} \frac{\operatorname{d}\!\varrho}{\varrho},\nonumber\\
  &\leq& \frac{2^{\frac{n}{2}}}{\ln2}\int_{R}^{2R}\left(\fint_{B_{\varrho}(x_{0})}\left|{\bf
F}-\overline{\bf
F}_{B_{\varrho}(x_{0})}\right|^{2}\operatorname{d}\!x\right)^{\frac{1}{2}}\frac{\operatorname{d}\!\varrho}{\varrho}\nonumber\\
&&+\frac{1}{\sigma^{\frac{n}{2}}\ln\frac{1}{\sigma}}
\sum_{i=1}^{\infty}\int_{\sigma^{i}R}^{\sigma^{i-1}R}\left(\fint_{B_{\varrho}(x_{0})}\left|{\bf
F}-\overline{\bf
F}_{B_{\varrho}(x_{0})}\right|^{2}\operatorname{d}\!x\right)^{\frac{1}{2}} \frac{\operatorname{d}\!\varrho}{\varrho},\nonumber\\
&\leq&C(n,\lambda,\Lambda)\int_{0}^{2R}\left( \fint_{B_{\varrho}(x_{0})}\left(\frac{|{\bf F}-\overline{{\bf F}}_{B_{\varrho}(x_{0})}|}{\varrho}\right)^{2}\operatorname{d}\!x\right)^{\frac{1}{2}}\operatorname{d}\!\varrho.
\end{eqnarray*}
Note that $[{\bf A}]^{2}_{\operatorname{ BMO}}(R)$ is a non-decreasing function with respect to $R$, then we derive
\begin{eqnarray}\label{th1-inequ4}
  && \sum_{i=1}^{k}\left(\fint_{B_{\sigma^{i}\!R}(x_{0})}\left|D{\bf
u}-\overline{D{\bf
u}}_{B_{\sigma^{i}\!R}(x_{0})}\right|+\left|{\pi}-\overline{\pi}_{B_{\sigma^{i}\!R}(x_{0})}\right|\operatorname{d}\!x\right)\nonumber\\
&\leq &\left(\frac{1}{4}+C_{1}\left([{\bf A}]_{\operatorname{ BMO}}^{2}(R)\right)^{\frac{1}{2}-\frac{1}{q}}\right)\sum_{i=0}^{k-1}\fint_{B_{\sigma^{i}R}(x_{0})}\left|D{\bf
u}-\overline{D {\bf
u}}_{B_{\sigma^{i}R}(x_{0})}\right|\operatorname{d}\!x\nonumber\\
&& + C\int_{0}^{2R}\left( \fint_{B_{\varrho}(x_{0})}\left(\frac{|{\bf F}-\overline{{\bf F}}_{B_{\varrho}(x_{0})}|}{\varrho}\right)^{2}\operatorname{d}\!x\right)^{\frac{1}{2}}\operatorname{d}\!\varrho\nonumber\\
&&+C\sum_{i=0}^{k-1}\left([{\bf A}]_{\operatorname{ BMO}}^{2}(\sigma^{i}R)\right)^{\frac{1}{2}-\frac{1}{q}}\left|\fint_{B_{\sigma^{i}R}(x_{0})}D{\bf
u}\operatorname{d}\!x\right|.
\end{eqnarray}
Now we choose the radius $R_{0}$ and $\delta>0$ such that
\begin{equation*}
  C_{1}\left([{\bf A}]_{\operatorname{ BMO}}^{2}(R)\right)^{\frac{1}{2}-\frac{1}{q}}\leq C_{1}\delta^{1-\frac{2}{q}}=\frac{1}{4}
\end{equation*}
for any $R\leq R_{0}$. Then the first term on the right side of \eqref{th1-inequ4} can be absorbed by the left side that
\begin{eqnarray}\label{th1-inequ5}
  && \sum_{i=1}^{k}\left(\fint_{B_{\sigma^{i}\!R}(x_{0})}\left|D{\bf
u}-\overline{D{\bf
u}}_{B_{\sigma^{i}\!R}(x_{0})}\right|+\left|{\pi}-\overline{\pi}_{B_{\sigma^{i}\!R}(x_{0})}\right|\operatorname{d}\!x\right)\nonumber\\
&\leq &\fint_{B_{R}(x_{0})}\left|D{\bf
u}-\overline{D {\bf
u}}_{B_{R}(x_{0})}\right|\operatorname{d}\!x + C\int_{0}^{2R}\left( \fint_{B_{\varrho}(x_{0})}\left(\frac{|{\bf F}-\overline{{\bf F}}_{B_{\varrho}(x_{0})}|}{\varrho}\right)^{2}\operatorname{d}\!x\right)^{\frac{1}{2}}\operatorname{d}\!\varrho\nonumber\\
&&+C\sum_{i=0}^{k-1}\left([{\bf A}]_{\operatorname{ BMO}}^{2}(\sigma^{i}R)\right)^{\frac{1}{2}-\frac{1}{q}}\left|\fint_{B_{\sigma^{i}R}(x_{0})}D{\bf
u}\operatorname{d}\!x\right|.
\end{eqnarray}

Next, we turn our attention to the estimate of the last term on the right side of \eqref{th1-inequ5}. A combination of \eqref{th1-inequ5} and \eqref{eqn-minimal2} yields that
\begin{eqnarray}\label{th1-inequ6}
 && \left|\fint_{B_{\sigma^{k+1}R}(x_{0})}D{\bf
u}\operatorname{d}\!x\right| \nonumber\\
  &=& \left|\sum_{i=0}^{k}\left(\fint_{B_{\sigma^{i+1}R}(x_{0})}D{\bf
u}\operatorname{d}\!x-\fint_{B_{\sigma^{i}R}(x_{0})}D{\bf
u}\operatorname{d}\!x\right)+\fint_{B_{R}(x_{0})}D{\bf
u}\operatorname{d}\!x\right|,\nonumber \\
  &\leq& \sigma^{-n}\sum_{i=0}^{k}\fint_{B_{\sigma^{i}R}(x_{0})}\left|D{\bf
u}-\overline{D{\bf u}}_{B_{\sigma^{i}R}(x_{0})}\right|\operatorname{d}\!x+\left|\fint_{B_{R}(x_{0})}D{\bf
u}\operatorname{d}\!x\right|,\nonumber\\
   &\leq& \frac{2}{\sigma^{n}}\fint_{B_{R}(x_{0})}\left|D{\bf
u}-\overline{D {\bf
u}}_{B_{R}(x_{0})}\right|\operatorname{d}\!x + C\int_{0}^{2R}\left( \fint_{B_{\varrho}(x_{0})}\left(\frac{|{\bf F}-\overline{{\bf F}}_{B_{\varrho}(x_{0})}|}{\varrho}\right)^{2}\operatorname{d}\!x\right)^{\frac{1}{2}}\operatorname{d}\!\varrho\nonumber\\
&&+C\sum_{i=0}^{k-1}\left([{\bf A}]_{\operatorname{ BMO}}^{2}(\sigma^{i}R)\right)^{\frac{1}{2}-\frac{1}{q}}\left|\fint_{B_{\sigma^{i}R}(x_{0})}D{\bf
u}\operatorname{d}\!x\right|+\left|\fint_{B_{R}(x_{0})}D{\bf
u}\operatorname{d}\!x\right|,\nonumber\\
&\leq& C\fint_{B_{R}(x_{0})}\left|D{\bf
u}\right|\operatorname{d}\!x + C\int_{0}^{2R}\left( \fint_{B_{\varrho}(x_{0})}\left(\frac{|{\bf F}-\overline{{\bf F}}_{B_{\varrho}(x_{0})}|}{\varrho}\right)^{2}\operatorname{d}\!x\right)^{\frac{1}{2}}\operatorname{d}\!\varrho\nonumber\\
&&+C\sum_{i=0}^{k-1}\left([{\bf A}]_{\operatorname{ BMO}}^{2}(\sigma^{i}R)\right)^{\frac{1}{2}-\frac{1}{q}}\left|\fint_{B_{\sigma^{i}R}(x_{0})}D{\bf
u}\operatorname{d}\!x\right|.
\end{eqnarray}
Setting
\begin{equation*}
  M:=\fint_{B_{R}(x_{0})}\left|D{\bf
u}\right|\operatorname{d}\!x + \int_{0}^{2R}\left( \fint_{B_{\varrho}(x_{0})}\left(\frac{|{\bf F}-\overline{{\bf F}}_{B_{\varrho}(x_{0})}|}{\varrho}\right)^{2}\operatorname{d}\!x\right)^{\frac{1}{2}}\operatorname{d}\!\varrho\,,
\end{equation*}
we claim that
\begin{equation}\label{th1-inequ7}
  \left|\fint_{B_{\sigma^{k+1}R}(x_{0})}D{\bf
u}\operatorname{d}\!x\right|\leq C(n,\lambda,\Lambda)M
\end{equation}
for any $k\in\mathbb{N}$. The proof of this claim is based on induction. Using \eqref{eqn-minimal2}, we have the following estimate for the case of $k=0$.
\begin{eqnarray*}
  \left|\fint_{B_{\sigma R}(x_{0})}D{\bf
u}\operatorname{d}\!x\right| &=& \left|\fint_{B_{\sigma R}(x_{0})}D{\bf
u}\operatorname{d}\!x-\fint_{B_{ R}(x_{0})}D{\bf
u}\operatorname{d}\!x+\fint_{B_{R}(x_{0})}D{\bf
u}\operatorname{d}\!x\right|, \\
 &\leq& \frac{\left|B_{R}\right|}{\left|B_{\sigma R}\right|}\fint_{B_{R}(x_{0})}\left|D{\bf
u}-\overline{D{\bf
u}}_{B_{R}(x_{0})}\right|\operatorname{d}\!x+\fint_{B_{R}(x_{0})}\left|D{\bf
u}\right|\operatorname{d}\!x, \\
   &\leq& C(n,\lambda,\Lambda) \fint_{B_{R}(x_{0})}\left|D{\bf
u}\right|\operatorname{d}\!x\leq CM.
\end{eqnarray*}
Let us suppose that \eqref{th1-inequ7} is true for all $k\leq k_{0}$, and we need to verify the case of $k=k_{0}+1$.
In terms of \eqref{th1-inequ6}, we get
\begin{eqnarray*}
  \left|\fint_{B_{\sigma^{k_{0}+2}R}(x_{0})}D{\bf
u}\operatorname{d}\!x\right|&\leq& CM+ C\sum_{i=0}^{k_{0}}\left([{\bf A}]_{\operatorname{ BMO}}^{2}(\sigma^{i}R)\right)^{\frac{1}{2}-\frac{1}{q}}\left|\fint_{B_{\sigma^{i}R}(x_{0})}D{\bf
u}\operatorname{d}\!x\right|,\\
   &\leq& CM+CM\sum_{i=0}^{k_{0}}\left([{\bf A}]_{\operatorname{ BMO}}^{2}(\sigma^{i}R)\right)^{\frac{1}{2}-\frac{1}{q}}.
\end{eqnarray*}
Applying the fact that $[{\bf A}]_{\operatorname{BMO}}^{2}(\cdot)$ is non-deceasing, $\sigma\in\left(0,\frac{1}{2}\right]$ and the definition of $d(\cdot)$ in \eqref{dini-bmo}, we obtain
\begin{eqnarray*}
&&\sum_{i=0}^{k_{0}}\left([{\bf A}]_{\operatorname{BMO}}^{2}(\sigma^{i}R)\right)^{\frac{1}{2}-\frac{1}{q}} \nonumber\\
&\leq& \sum_{i=0}^{\infty}\left([{\bf A}]_{\operatorname{ BMO}}^{2}(\sigma^{i}R)\right)^{\frac{1}{2}-\frac{1}{q}},\nonumber \\
   &=& \frac{1}{\ln 2}\int_{R}^{2R}\left([{\bf A}]_{\operatorname{ BMO}}^{2}(R)\right)^{\frac{1}{2}-\frac{1}{q}}\frac{\operatorname{d}\!\varrho}{\varrho}+\frac{1}{\ln \frac{1}{\sigma}}\sum_{i=1}^{\infty}\int_{\sigma^{i}R}^{\sigma^{i-1}R}\left([{\bf A}]_{\operatorname{ BMO}}^{2}(\sigma^{i}R)\right)^{\frac{1}{2}-\frac{1}{q}}\frac{\operatorname{d}\!\varrho}{\varrho},\nonumber \\
   &\leq& \frac{1}{\ln 2}\int_{R}^{2R}\left([{\bf A}]_{\operatorname{ BMO}}^{2}(\varrho)\right)^{\frac{1}{2}-\frac{1}{q}}\frac{\operatorname{d}\!\varrho}{\varrho}+\frac{1}{\ln \frac{1}{\sigma}}\sum_{i=1}^{\infty}\int_{\sigma^{i}R}^{\sigma^{i-1}R}\left([{\bf A}]_{\operatorname{ BMO}}^{2}(\varrho)\right)^{\frac{1}{2}-\frac{1}{q}}\frac{\operatorname{d}\!\varrho}{\varrho},\nonumber \\
   &\leq& \left(\frac{1}{\ln 2}+\frac{1}{\ln \frac{1}{\sigma}}\right)\int_{0}^{2R}\left([{\bf A}]_{\operatorname{ BMO}}^{2}(\varrho)\right)^{\frac{1}{2}-\frac{1}{q}}\frac{\operatorname{d}\!\varrho}{\varrho},\nonumber\\
   &\leq&\frac{2}{\ln 2}d(2R).
\end{eqnarray*}
We now further restrict the value of $R_{0}$ such that
\begin{equation*}
  \frac{2}{\ln 2}d(2R_{0})\leq1.
\end{equation*}
By virtue of the fact that $d(\cdot)$ is non-deceasing, we have
\begin{equation}\label{th1-inequ8}
  \frac{2}{\ln 2}d(2R)\leq1
\end{equation}
for any $R\leq R_{0}$. Then the above estimates infer that
\begin{equation*}
  \left|\fint_{B_{\sigma^{k_{0}+2}R}(x_{0})}D{\bf
u}\operatorname{d}\!x\right|\leq CM.
\end{equation*}
Hence, the claim \eqref{th1-inequ7} holds whenever $k\in\mathbb{N}$.
Passing to the limit as $k\rightarrow\infty$ in \eqref{th1-inequ5} and utilizing \eqref{th1-inequ7}, \eqref{th1-inequ8}, \eqref{eqn-minimal2} and the definition of $M$, we deduce
\begin{eqnarray}\label{th1-inequ9}
  && \sum_{i=1}^{\infty}\left(\fint_{B_{\sigma^{i}\!R}(x_{0})}\left|D{\bf
u}-\overline{D{\bf
u}}_{B_{\sigma^{i}\!R}(x_{0})}\right|+\left|{\pi}-\overline{\pi}_{B_{\sigma^{i}\!R}(x_{0})}\right|\operatorname{d}\!x\right)\nonumber\\
&\leq &\fint_{B_{R}(x_{0})}\left|D{\bf
u}-\overline{D {\bf
u}}_{B_{R}(x_{0})}\right|\operatorname{d}\!x + C\int_{0}^{2R}\left( \fint_{B_{\varrho}(x_{0})}\left(\frac{|{\bf F}-\overline{{\bf F}}_{B_{\varrho}(x_{0})}|}{\varrho}\right)^{2}\operatorname{d}\!x\right)^{\frac{1}{2}}\operatorname{d}\!\varrho+CM,\nonumber\\
&\leq&C\fint_{B_{R}(x_{0})}\left|D{\bf
u}\right|\operatorname{d}\!x + C\int_{0}^{2R}\left( \fint_{B_{\varrho}(x_{0})}\left(\frac{|{\bf F}-\overline{{\bf F}}_{B_{\varrho}(x_{0})}|}{\varrho}\right)^{2}\operatorname{d}\!x\right)^{\frac{1}{2}}\operatorname{d}\!\varrho\,.
\end{eqnarray}

In the sequel, let $k\rightarrow\infty$ in \eqref{th1-inequ1} and \eqref{th1-inequ2}, meanwhile applying the Lebesgue differentiation Theorem and \eqref{th1-inequ9}, we conclude
\begin{eqnarray*}
   &&  \left|D{\bf u}(x_{0})-\fint_{B_{R}(x_{0})}D{\bf u}\operatorname{d}\!x \right|+ \left|\pi(x_{0})-\fint_{B_{R}(x_{0})}\pi\operatorname{d}\!x \right| \\
   &\leq& \sigma^{-n}\sum_{i=0}^{\infty}\left(\fint_{B_{\sigma^{i}\!R}(x_{0})}\left|D{\bf
u}-\overline{D{\bf
u}}_{B_{\sigma^{i}\!R}(x_{0})}\right|+\left|{\pi}-\overline{\pi}_{B_{\sigma^{i}\!R}(x_{0})}\right|\operatorname{d}\!x\right) \\
  &\leq& C\fint_{B_{R}(x_{0})}\left|D{\bf
u}\right|\operatorname{d}\!x +C\fint_{B_{R}(x_{0})}\left|\pi\right|\operatorname{d}\!x+ C\int_{0}^{2R}\left( \fint_{B_{\varrho}(x_{0})}\left(\frac{|{\bf F}-\overline{{\bf F}}_{B_{\varrho}(x_{0})}|}{\varrho}\right)^{2}\operatorname{d}\!x\right)^{\frac{1}{2}}\operatorname{d}\!\varrho
\end{eqnarray*}
for almost every $x_{0}\in\Omega$. As a consequence of the above inequalities, we derive
the potential gradient estimate
\begin{equation*}
  |D{\bf u}(x_{0})|+|\pi (x_{0})|\leq C\fint_{B_{R}(x_{0})}\left|D{\bf u}\right|\operatorname{d}\!x+ C\fint_{B_{R}(x_{0})}\left|\pi\right|\operatorname{d}\!x+C\int_{0}^{2R}\left( \fint_{B_{\varrho}(x_{0})}\left(\frac{|{\bf F}-\overline{{\bf F}}_{B_{\varrho}(x_{0})}|}{\varrho}\right)^{2}\operatorname{d}\!x\right)^{\frac{1}{2}}\operatorname{d}\!\varrho
\end{equation*}
holds for almost every $x_{0}\in\Omega$ and every $B_{2R}(x_{0})\subset\Omega$ with $R\leq R_{0}$,
where the positive constant $C=C(n,\lambda,\Lambda)$ and the positive radius $R_{0}=R_{0}\left(n,\lambda,\Lambda,d(\cdot)\right)$.
This completes the proof of Theorem \ref{Th1}\,.
\end{proof}

Subsequently, we will be devoted to the proof of Theorem \ref{Th2} with respect to the zero order pointwise estimate. Note that we only need to impose the uniformly elliptic and bounded assumptions on the vector field $\bf A$ from now on.

\begin{proof}[Proof of Theorem \ref{Th2}] By virtue of \cite{GMo}, we have the following Caccioppoli inequality of the weak solution $\bf v$ to \eqref{ComparisonSystem} that
\begin{equation}\label{th2-inequ1}
  \fint_{B_{R}(x_{0})}|D{\bf v}|^{2}\operatorname{d}\!x\leq\frac{C(n,\lambda,\Lambda)}{R^{2}}\fint_{B_{2R}(x_{0})}|{\bf v}-\overline{\bf v}_{B_{2R}(x_{0})}|^{2}\operatorname{d}\!x\,.
\end{equation}
Thus, a combination of the Sobolev embedding Theorem and \eqref{th2-inequ1} yields that
\begin{equation*}
  \left(\fint_{B_{R}(x_{0})}\left|{\bf
v}-\overline{{\bf
v}}_{B_{R}(x_{0})}\right|^{s}\operatorname{d}\!x\right)^{\frac{1}{s}}\leq C(n,\lambda,\Lambda)\left(\fint_{B_{2R}(x_{0})}\left|{\bf
v}-\overline{{\bf
v}}_{B_{2R}(x_{0})}\right|^{2}\operatorname{d}\!x\right)^{\frac{1}{2}},
\end{equation*}
where
\begin{equation*}
s=
  \left\{\begin{array}{r@{\ \ }c@{\ \ }ll}
\frac{2n}{n-2},  \ & n&>2, \\
p,  \ & n&=2,
\end{array}\right.
\end{equation*}
for any $p\in[2,+\infty)$. Then it follows from Lemma \ref{reverse-holder} and H\"{o}lder's inequality that
\begin{equation}\label{th2-inequ2}
  \left(\fint_{B_{R}(x_{0})}\left|{\bf
v}-\overline{{\bf
v}}_{B_{R}(x_{0})}\right|^{2}\operatorname{d}\!x\right)^{\frac{1}{2}}\leq C(n,\lambda,\Lambda)\fint_{B_{2R}(x_{0})}\left|{\bf
v}-\overline{{\bf
v}}_{B_{2R}(x_{0})}\right|\operatorname{d}\!x.
\end{equation}
On the other hand, utilizing \eqref{th2-inequ1} and \cite[Theorem 7.7]{Giu}, we obtain the decay estimate for $\bf v$ as follows
\begin{equation}\label{th2-inequ3}
  \fint_{B_{\sigma R}(x_{0})}\left|{\bf
v}-\overline{{\bf
v}}_{B_{\sigma R}(x_{0})}\right|^{2}\operatorname{d}\!x\leq C(n,\lambda,\Lambda)\sigma^{2\gamma}\fint_{B_{R}(x_{0})}\left|{\bf
v}-\overline{{\bf
v}}_{B_{R}(x_{0})}\right|^{2}\operatorname{d}\!x
\end{equation}
for any $\sigma\in (0,1]$ and some $\gamma=\gamma(n,\lambda,\Lambda)>0$.

Next, we aim to derive a Campanato-type decay estimate for the weak solution $\bf u$ to \eqref{model}. Applying H\"{o}lder's inequality, \eqref{eqn-minimal2}, \eqref{u-v}, \eqref{th2-inequ3} and \eqref{th2-inequ2}, we deduce
\begin{eqnarray*}
  && \left(\fint_{B_{\sigma\!R}(x_{0})}\left|{\bf
u}-\overline{{\bf
u}}_{B_{\sigma\!R}(x_{0})}\right|\operatorname{d}\!x\right)^{2}\\
   &\leq& \fint_{B_{\sigma\!R}(x_{0})}\left|{\bf
u}-\overline{{\bf
v}}_{B_{\sigma\!R}(x_{0})}\right|^{2}\operatorname{d}\!x\,, \\
  &\leq& C(n)\sigma^{-n}\fint_{B_{2R}(x_{0})}\left|{\bf
u}-{\bf
v}\right|^{2}\operatorname{d}\!x+2\fint_{B_{\sigma\!R}(x_{0})}\left|{\bf
v}-\overline{{\bf
v}}_{B_{\sigma\!R}(x_{0})}\right|^{2}\operatorname{d}\!x, \\
   &\leq&  C\sigma^{-n}R^{2}\fint_{B_{2R}(x_{0})}\!\left|{\bf F}-\overline{{\bf F}}_{B_{2R}(x_{0})}\right|^{2}\operatorname{d}\!x+ C\sigma^{2\gamma}\fint_{B_{R}(x_{0})}\left|{\bf
v}-\overline{{\bf
v}}_{B_{R}(x_{0})}\right|^{2}\operatorname{d}\!x, \\
  &\leq&  C\sigma^{-n}R^{2}\fint_{B_{2R}(x_{0})}\!\left|{\bf F}\right|^{2}\operatorname{d}\!x+ C\sigma^{2\gamma}\left(\fint_{B_{2R}(x_{0})}\left|{\bf
v}-\overline{{\bf
v}}_{B_{2R}(x_{0})}\right|\operatorname{d}\!x\right)^{2},\\
&\leq&  C\sigma^{-n}R^{2}\fint_{B_{2R}(x_{0})}\!\left|{\bf F}\right|^{2}\operatorname{d}\!x+ C\sigma^{2\gamma}\left(\fint_{B_{2R}(x_{0})}\left|{\bf
u}-{\bf
v}\right|+\left|{\bf
u}-\overline{{\bf
u}}_{B_{2R}(x_{0})}\right|\operatorname{d}\!x\right)^{2},\\
&\leq& C\sigma^{2\gamma}\left(\fint_{B_{2R}(x_{0})}\left|{\bf
u}-\overline{{\bf
u}}_{B_{2R}(x_{0})}\right|\operatorname{d}\!x\right)^{2}+ C_{\sigma}R^{2}\fint_{B_{2R}(x_{0})}\!\left|{\bf F}\right|^{2}\operatorname{d}\!x,
\end{eqnarray*}
where $C=C(n,\lambda,\Lambda)$ and $C_{\sigma}=C(n,\lambda,\Lambda,\sigma)$. Hence, we get the following Campanato-type decay estimate with respect to $\bf u$
\begin{equation*}
  \fint_{B_{\sigma\!R}(x_{0})}\left|{\bf
u}-\overline{{\bf
u}}_{B_{\sigma\!R}(x_{0})}\right|\operatorname{d}\!x\leq C\sigma^{\gamma}\fint_{B_{2R}(x_{0})}\left|{\bf
u}-\overline{{\bf
u}}_{B_{2R}(x_{0})}\right|\operatorname{d}\!x+ C_{\sigma}R\left(\fint_{B_{2R}(x_{0})}\!\left|{\bf F}\right|^{2}\operatorname{d}\!x\right)^{\frac{1}{2}}.
\end{equation*}
The subsequent proof goes exactly as that of Theorem \ref{Th1}\,, we only need to reestimate the integral term involving $\bf F$ as follows
\begin{eqnarray*}
  && \sum_{i=0}^{k-1}\sigma^{i}R\left(\fint_{B_{\sigma^{i}R}(x_{0})}\left|{\bf
F}\right|^{2}\operatorname{d}\!x\right)^{\frac{1}{2}}\nonumber \\
   &\leq&  R\left(\fint_{B_{R}(x_{0})}\left|{\bf
F}\right|^{2}\operatorname{d}\!x\right)^{\frac{1}{2}}+ \sum_{i=1}^{\infty}\sigma^{i}R\left(\fint_{B_{\sigma^{i}R}(x_{0})}\left|{\bf
F}\right|^{2}\operatorname{d}\!x\right)^{\frac{1}{2}}, \nonumber\\
   &=&\frac{1}{\ln2}\int_{R}^{2R}R\left(\fint_{B_{R}(x_{0})}\left|{\bf
F}\right|^{2}\operatorname{d}\!x\right)^{\frac{1}{2}}\frac{\operatorname{d}\!\varrho}{\varrho}+\frac{1}{\ln\frac{1}{\sigma}}
\sum_{i=1}^{\infty}\int_{\sigma^{i}R}^{\sigma^{i-1}R}\sigma^{i}R\left(\fint_{B_{\sigma^{i}R}(x_{0})}\left|{\bf
F}\right|^{2}\operatorname{d}\!x\right)^{\frac{1}{2}} \frac{\operatorname{d}\!\varrho}{\varrho},\nonumber\\
  &\leq& \frac{2^{\frac{n}{2}}}{\ln2}\int_{R}^{2R}\varrho\left(\fint_{B_{\varrho}(x_{0})}\left|{\bf
F}\right|^{2}\operatorname{d}\!x\right)^{\frac{1}{2}}\frac{\operatorname{d}\!\varrho}{\varrho}+\frac{1}{\sigma^{\frac{n}{2}}\ln\frac{1}{\sigma}}
\sum_{i=1}^{\infty}\int_{\sigma^{i}R}^{\sigma^{i-1}R}\varrho\left(\fint_{B_{\varrho}(x_{0})}\left|{\bf
F}\right|^{2}\operatorname{d}\!x\right)^{\frac{1}{2}} \frac{\operatorname{d}\!\varrho}{\varrho},\nonumber\\
&\leq&C(n,\lambda,\Lambda)\int_{0}^{2R}\left( \fint_{B_{\varrho}(x_{0})}\left|{\bf F}\right|
^{2}\operatorname{d}\!x\right)^{\frac{1}{2}}\operatorname{d}\!\varrho\nonumber\\
&=&C\,{\bf W}_{\frac{2}{3},3}^{2R}\left(\left|{\bf F}\right|
^{2}\right)(x_{0})\,.
\end{eqnarray*}
Therefore, we can deduce the desired zero order pointwise estimate \eqref{zero-estimate}, which completes the proof of Theorem \ref{Th2}\,.
\end{proof}

\section*{Acknowledgments}The authors are supported by the National Natural Science Foundation of China (NNSF  Grant No.11671414 and No.11771218).

\bibliography{bibliography}

\begin{thebibliography}{99}
\bibitem{ADM}\label{ADM} G.~Acosta, R. G.~Dur\'an, and M. A.~Muschietti, \textit{Solutions of the divergence operator on John domains}, Adv. Math. 206 (2006), 373--401.
\bibitem{BMR} \label{BMR} M.~Bul\'{\i}\u{c}ek, J.~M\'{a}lek, and K. R.~Rajagopal, \textit{Navier's slip and evolutionary Navier-Stokes-like systems with pressure and shear-rate dependent viscosity}, Indiana Univ. Math. J. 56 (2007), 51--85.
\bibitem{ChLe} \label{ChLe} J.~Choi and K.A.~Lee, \textit{The Green function for the Stokes system with measurable coefficients}, Commun. Pure Appl. Anal. 16 (2017), 1989--2022.
\bibitem{CiSc} \label{CiSc} A.~Cianchi and S.~Schwarzacher, \textit{Potential estimates for the p-Laplace system with data in divergence form}, J. Differential Equations 265 (2018), 478--499.
\bibitem{DiKaSc} \label{DiKaSc} L.~Diening, P.~Kaplick\'y, and S.~Schwarzacher, \textit{Campanato estimates for the generalized Stokes system}, Ann. Mat. Pura Appl. 193 (2014), 1779--1794.
\bibitem{Dong} \label{Dong} H.~Dong and D.~Kim, \textit{$L_{q}$-estimates for stationary Stokes system with coefficients measurable in one direction}, Bull. Math. Sci. (2018), doi:10.1007/s13373-018-0120-6.
\bibitem{DuMi} \label{DuMi} F.~Duzaar and G.~Mingione, \textit{Gradient continuity estimates}, Calc. Var. Partial Differential Equations 39 (2010), 379--418.
\bibitem{DuMi2} \label{DuMi2} F.~Duzaar and G.~Mingione, \textit{Gradient estimates via linear and nonlinear potentials}, J. Funct. Anal. 259 (2010), 2961--2998.
\bibitem{DuMi3} \label{DuMi3} F.~Duzaar and G.~Mingione, \textit{Gradient estimates via non-linear potentials}, Amer. J. Math. 133 (2011), 1093--1149.
\bibitem{GMo}\label{GMo} M.~Giaquinta and G.~Modica, \textit{Non linear systems of the type of the stationary Navier-Stokes system}, J. Reine Angew. Math. 330 (1982), 173--214.
\bibitem{Giu} \label{Giu} E.~Giusti, \textit{Direct Methods in the Calculus of Variations},  World Scientific, Singapore, 2003.
\bibitem{HeWo} \label{HeWo} L.~Hedberg and Th.H.~Wolff,  \textit{Thin sets in Nonlinear Potential Theory}, Ann. Inst.
Fourier (Grenoble) 33 (1983), 161--187.
\bibitem{KiMa} \label{KiMa} T.~Kilpel\"{a}inen and J.~Mal\'{y}, \textit{Degenerate elliptic equations with measure data and nonlinear potentials},
Ann. Scuola Norm. Sup. Pisa Cl. Sci. 19 (1992), 591--613.
\bibitem{KiMa2} \label{KiMa2} T.~Kilpel\"{a}inen and J.~Mal\'{y}, \textit{The Wiener test and potential estimates for quasilinear elliptic equations},
Acta Math. 172 (1994), 137--161.
\bibitem{KuMi} \label{KuMi} T.~Kuusi and G.~Mingione, \textit{Universal potential estimates}, J. Funct. Anal. 262 (2012), 4205--4269.
\bibitem{KuMi2} \label{KuMi2} T.~Kuusi and G.~Mingione, \textit{A nonlinear Stein theorem}, Calc. Var. Partial Differential Equations 51 (2014), 45--86.
\bibitem{KuMi3} \label{KuMi3} T.~Kuusi and G.~Mingione, \textit{Vectorial nonlinear potential theory}, J. Eur. Math. Soc. 20 (2018), 929--1004.
\bibitem{Lab} \label{Lab} D.~Labutin, \textit{Potential estimates for a class of fully nonlinear elliptic equations}, Duke Math. J. 111 (2002), 1--49.
\bibitem{La} \label{La} O.A.~Ladyzhenskaya, \textit{The mathematical theory of viscous incompressible flow}, Gordon and Breach, Science Publishers, New York-London-Paris (1969).
\bibitem{MaHa} \label{MaHa} V.G.~Maz'ja and V.P.~Havin, \textit{A nonlinear potential theory}, Uspehi Mat. Nauk 27 (1972), 67--138.
\bibitem{Min} \label{Min} G.~Mingione, \textit{Gradient potential estimates}, J. Eur. Math. Soc. 13 (2011), 459--486.
\bibitem{TrWa} \label{TrWa} N.S.~Trudinger and X.J.~Wang, \textit{On the weak continuity of elliptic operators and applications to potential theory}, Amer. J. Math. 124 (2002), 369--410.

\end{thebibliography}

\end{document}